\documentclass[11pt,thmsa]{article}

\usepackage{amssymb,latexsym}
\usepackage{amsmath,amscd}
\usepackage[all,knot]{xy}

 \usepackage[latin1]{inputenc}

 \usepackage{mathrsfs}
 \usepackage{amsmath}
 \usepackage{amsthm}
 \usepackage{amsfonts}
 \usepackage{amssymb}
 \usepackage[pdftex]{graphicx}
 \usepackage{wrapfig}
 \usepackage{layout}
 \usepackage{verbatim}
 \usepackage{alltt}
 \usepackage{xypic}
 \usepackage{xcolor}
 \usepackage{graphics}
 \input xy
 \xyoption{all}

 \DeclareMathAlphabet{\mathpzc}{OT1}{pzc}{m}{it}

 \newtheorem{theorem}{Theorem}[section]
 \newtheorem{lemma}[theorem]{Lemma}

 \newtheorem{proposition}[theorem]{Proposition}
 
 \newtheorem{corollary}[theorem]{Corollary}
 \newtheorem{definition}[theorem]{Definition}
  \theoremstyle{definition}
 \newtheorem{example}[theorem]{Example}

 \newtheorem{remark}[theorem]{Remark}

\newtheorem*{acknowledgements}{Acknowledgements}
\renewenvironment{proof}{\noindent{\it
Proof.}}{\bgroup\hspace{\stretch{1}}$\square$\egroup\medskip\par}

\newcommand{\ad}{\mathrm{ad}}

\newcommand{\id}{\mathrm{id}}
\newcommand{\End}{\mathrm{End}}

\newcommand{\Aut}{\mathsf{Aut}}
\newcommand{\g}{\frak{g}}
\newcommand{\U}{\mathsf{U}({\frak{g}})}
\newcommand{\slt}{\mathfrak{sl}_2(\mathbb{C})}
\newcommand{\h}{\mathfrak{h}}
\begin{document}

\vspace{15cm}
 \title{Introduction to representations of braid groups}
\author{Camilo Arias Abad\footnote{Max Planck Institut f\"ur Mathematik, Vivatgasse 7, Bonn,
camiloariasabad@gmail.com. Supported
by SNF Grant 200020-131813/1 and a Humboldt fellowship.} \hspace{0cm}}

 \maketitle
 \begin{abstract} 
These are lecture notes prepared for a minicourse given at the Cimpa Research School \textit{Algebraic and geometric
aspects of representation theory}, held in Curitiba, Brazil in March 2013. The purpose of the course is to 
provide an introduction to the study of representations of braid groups. Three general classes of representations of braid groups are considered: homological representations via mapping class groups, monodromy representations via the Knizhnik-Zamolodchikov connection, and solutions of the Yang-Baxter equation via braided bialgebras. Some of the remarkable relations between these three different constructions are described.
\end{abstract}

\tableofcontents
\acknowledgements{I would like to thank Florian Sch\"atz and David Martinez Torres  for their comments on an earlier version, and Bernardo Uribe, for his help with some of the references. I would also like to thank the organizers of the Cimpa School for the invitation to teach this course, which gave me the opportunity to learn about braid groups. 
I am very grateful to the anonymous referee for carefully reading a previous version and correcting several mistakes.}
\section{Introduction}

Braid groups  were introduced by Emil Artin  in 1925, and by now play a role in various parts
of mathematics including knot theory, low dimensional topology, public key cryptography and deformation quantization.
The braid group $B_n$ admits different equivalent definitions: the fundamental group of the configuration space $C_n(\mathbb{C})$ of $n$ unordered points in the complex plane, the mapping class group of a disk with $n$ marked points, and Artin's presentation.
Each of these descriptions provides a method for constructing representations of $B_n$: fundamental groups act via monodromy, mapping class groups act on the homology groups of spaces, and solutions to the Yang-Baxter equations
produce solutions to Artin's relations.
The aim of these notes is to introduce these methods, and describe some of the relations between these seemingly unrelated constructions. The notes are organized as follows:

In section $\S \ref{section2}$ we present three different ways in which the braid group $B_n$ can be defined, explain how these definitions are equivalent, and introduce some general facts regarding the braid groups.
Section $\S \ref{section3}$ regards the group $B_n$ as the mapping class group of a disk with $n$ marked points. 
This description produces representations of $B_n$ via the action on the homology of spaces which are functorially associated to the punctured disk. The Burau representation is discussed and the relationship with the Alexander-Conway polynomial is explained. The Lawrence-Krammer-Bigelow representation is also constructed.
Section $\S \ref{section4}$ begins with the introduction of the Knizhnik-Zamolodchikov connection and explains how representations of complex semisimple Lie algebras produce, via the Knizhnik-Zamolodchivok connection, monodromy representations of the braid group. A theorem of Kohno is explained, which describes the relationship between monodromy representations associated to the Lie algebra $\slt$ and the Lawrence-Krammer-Bigelow representation.
In section $\S \ref{section5}$ we introduce the Yang-Baxter equation, and explain the way in which solutions of this equation produce representations of $B_n$. We define the notion of a quasi-triangular bialgebra, and show that modules over quasi-triangular bialgebras come equipped with solutions to the Yang-Baxter equation. Finally, the Drinfeld-Kohno theorem is discussed. This remarkable theorem explains the precise relationship between the monodromy representations constructed via the Knizhnik-Zamolodchikov connection, and the representations constructed as solutions of the Yang-Baxter equation provided by quantum enveloping algebras.
We have included some proofs of the results we discussed, when they are enlightening and sufficiently simple.
However, no attempt has been made to present all proofs of the theorems we mentioned, as this would go far beyond
the scope of these notes. When proofs are missing, we have tried to provide references to the many excellent sources that were used in preparing this course.
\section{Braid groups}\label{section2}

Artin's braid groups admit several equivalent descriptions, which we will review in this section.
\subsection{Artin's presentation}
We denote by $D_2$ the unit disk in the complex plane and fix $n$ marked points \[-1< p_1 < p_2 \dots < p_n<1 \in D_2 \cap \mathbb{R}.\]
\begin{definition}
A braid is a collection of $n$  paths $f_i: I =[0,1] \mapsto D_2 $, called strands, such that:
\begin{enumerate}
\item $f_i(0)=p_i$.
\item $f_i(1)= p_{\tau(i)}$, for some permutation $\tau$ in the symmetric group $\Sigma_n$.
\item For each $t \in I$, $f_i(t)\neq f_j(t)$ provided that $i\neq j$.
\end{enumerate}
\end{definition}

Braids are pictured geometrically as a collection of $n$ strands in three dimensional space. 
One is usually interested in braids only up to isotopies fixing the endpoints of each strand. 
We will abuse the notation and take the word braid to mean an isotopy class.
Some examples of braids are the following:

\begin{figure}[h!]
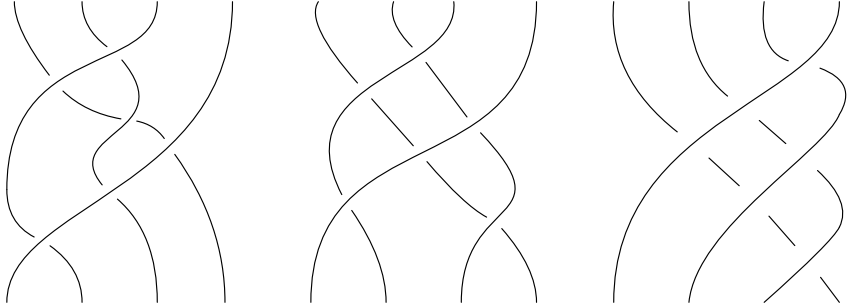

\[ \xy 
(-15,-20)*{}="b1"; (-5,-20)*{}="b2"; 
(5,-20)*{}="b3"; (14,-20)*{}="b4"; 
(-14,20)*{}="T1"; (-5,20)*{}="T2"; 
(5,20)*{}="T3"; (15,20)*{}="T4"; 
"b1"; "T4" **\crv{(-15,-7) & (15,-5)} 
\POS?(.25)*{\hole}="2x" \POS?(.47)*{\hole}="2y" \POS?(.65)*{\hole}="2z"; 
"b2";"2x" **\crv{(-5,-15)}; 
"b3";"2y" **\crv{(5,-10)}; 
"b4";"2z" **\crv{(14,-9)}; 
(-15,-5)*{}="3x"; 
"2x"; "3x" **\crv{(-15,-10)}; 
"3x"; "T3" **\crv{(-15,15) & (5,10)} 
\POS?(.38)*{\hole}="4y" \POS?(.65)*{\hole}="4z"; 
"T1";"4y" **\crv{(-14,16)}; 
"T2";"4z" **\crv{(-5,16)}; 
"2y";"4z" **\crv{(-10,3) & (10,2)} \POS?(.6)*{\hole}="5z"; 
"4y";"5z" **\crv{(-5,5)}; 
"5z";"2z" **\crv{(5,4)};   
\endxy
 \quad \,\,\,\, \quad
\xy 
(-14,20)*{}="T1"; (-4,20)*{}="T2"; 
(4,20)*{}="T3"; (15,20)*{}="T4"; 
(-15,-20)*{}="B1"; (-5,-20)*{}="B2"; 
(5,-20)*{}="B3"; (15,-20)*{}="B4"; 
"B1";"T4" **\crv{(-15,5) & (15,-5)} 
\POS?(.25)*{\hole}="2x" \POS?(.49)*{\hole}="2y" \POS?(.65)*{\hole}="2z"; 
"2x";"T3" **\crv{(-20,10) & (5,10) } 
\POS?(.45)*{\hole}="3y" \POS?(.7)*{\hole}="3z"; 
"2x";"B2" **\crv{(-5,-14)}; 
"T1";"3y" **\crv{(-16,17)}; 
"T2";"3z" **\crv{(-5,17)}; 
"3z";"2z" **\crv{}; 
"3y";"2y" **\crv{}; 
"B3";"2z" **\crv{ (5,-5) &(20,-10)} 
\POS?(.4)*{\hole}="4z"; 
"2y";"4z" **\crv{(6,-8)}; 
"4z";"B4" **\crv{(15,-15)}; 
\endxy
\quad \,\,\,\, \quad
\xy 
(15,20)*{}="T1"; (5,20)*{}="T2"; 
(-5,20)*{}="T3"; (-15,20)*{}="T4"; 
(15,-20)*{}="B1"; (5,-20)*{}="B2"; 
(-5,-20)*{}="B3"; (-15,-20)*{}="B4"; 
"T1"; "B4" **\crv{(15,7) & (-15,5)} 
\POS?(.25)*+{\hole}="2x" \POS?(.45)*+{\hole}="2y" \POS?(.6)*+{\hole}="2z"; 
"T2";"2x" **\crv{(4,12)}; 
"T3";"2y" **\crv{(-5,10)}; 
"T4";"2z" **\crv{(-16,9)}; 
(15,5)*{}="3x"; 
"2x"; "3x" **\crv{(18,10)}; 
"3x"; "B3" **\crv{(13,0) & (-4,-10)} 
\POS?(.3)*+{\hole}="4x" \POS?(.53)*+{\hole}="4y"; 
"2y"; "4x" **\crv{}; 
"2z"; "4y" **\crv{}; 
(15,-10)*{}="5x"; 
"4x";"5x" **\crv{(17,-6)}; 
"5x";"B2" **\crv{(14,-12)} 
\POS?(.6)*+{\hole}="6x"; 
"6x";"B1" **\crv{}; 
"4y";"6x" **\crv{}; 
\endxy
 \]
 \caption{Examples of braids}
\end{figure}

\begin{definition}
The braid group in $n$ strands, denoted $B_n$, consists of the set of all
braids with $n$ strands with multiplication given by glueing.
\end{definition}
 
It is a good exercise to become convinced that this operation indeed gives $B_n$ the structure of a group.
The following result, due to Emil Artin, gives a presentation for the group $B_n$.
\begin{theorem}[Artin \cite{Artin}]
Denote by $\mathcal{B}_n$ the group generated by the symbols $\sigma_1, \dots , \sigma_{n-1}$ modulo the 
relations:
\begin{enumerate}
\item $\sigma_i \sigma_{i+1}\sigma_i=\sigma_{i+1} \sigma_{i}\sigma_{i+1}$ for $i=1,...,n-2$.
\item $\sigma_i \sigma_j=\sigma_j \sigma_i$ for $|i-j|\geq2$.
\end{enumerate}

There is an isomorphism of groups $ \phi: \mathcal{B}_n \rightarrow B_n$ determined by:

\begin{figure}[h!]
\[ \sigma_i
    \quad \mapsto \quad
    \vcenter{\xy 0;/r1pc/:
    @={(0.5,0)}, @@{*{\vtwistneg}},
    @i@={(0,0),(0.5,0),(2.5,0),(3,0)} @@{="save";"save"-(0,1),**@{-}}
    \endxy}\,\,.\]
   
    \end{figure}
\end{theorem}


\begin{proof}
The fact that the first relation is satisfied follows from the equality:
\begin{figure}[h!]

\[\vcenter{\xy 0;/r1pc/:
    @={(0.5,-0.5),(0,0)}, @@{*{\vtwistneg}},(0,-2),\vtwistneg,
    @i@={(0,-1),(2,0),(2,-2)} @@{="save";"save"-(0,1),**@{-}}\endxy}
    \quad = \quad
    \vcenter{\xy 0;/r1pc/:
    @={(0.5,0),(0,-0.5),(0.5,-1)}, @@{*{\vtwistneg}},
    @i@={(0,0),(2,-1),(0,-2)} @@{="save";"save"-(0,1),**@{-}}
    \endxy}\]
    \end{figure}

The fact that the second relation is satisfied follows from:

\begin{figure}[h!]
\[\vcenter{\xy 0;/r1pc/:
    @={(0,0)}, @@{*{\vtwistneg}},(1.5, -1),\vtwistneg,
   @i@={(0,-1),(1,-1),(2.5,0),(1.5,0)}  @@{="save";"save"-(0,1),**@{-}}\endxy}
   \quad=\quad
   \vcenter{\xy 0;/r1pc/:
    @={(0,-0.5)}, @@{*{\vtwistneg}},(1.5, 0),\vtwistneg,
   @i@={(0,0),(1,0),(2.5,-1),(1.5,-1)}  @@{="save";"save"-(0,1),**@{-}}\endxy}\]
    \end{figure}
We conclude that $\phi$ defines a homomorphism.
Next, one needs to prove that $\phi$ is actually an isomorphism.
Surjectivity is easy since the elements $\phi(\sigma_i)$ clearly generate $B_n$. 
For a proof of injectivity see \cite{KT} Theorem 1.12.

\end{proof}

From now on we will identify the groups $B_n$ and $\mathcal{B}_n$ using the isomorphism above.

\subsection{Configuration spaces}

The group $B_n$ can also be described as the fundamental group of the configuration space of points in the
plane. We will now briefly recall some definitions regarding configuration spaces.
Let $M$ be a smooth manifold. The configuration space of $n$ ordered points in $M$, denoted $\hat{C}_n(M)$
is the manifold:

\[\hat{C}_n(M):= \{(x_1, \dots ,x_n)\in M^n : x_i \neq x_j \text{ if } i\neq j\}.\]

There is a natural action of the symmetric group $\Sigma_n$ on the space $\hat{C}_n(M)$, given by permuting the coordinates. The configuration space of $n$ unordered points in $M$ is the quotient space:

\[C_n(M):= \hat{C}_n(M)/\Sigma_n.\]

\begin{theorem}[Fox-Neuwirth \cite{FoxN}, Fadell-van Buskirk \cite{FadelVan}]
There is a natural isomorphism between the braid group in $n$ strands and the fundamental group of the
configuration space of $n$ unordered points in the plane:

\[B_n \cong \pi_1(C_n(\mathbb{C}), \bf{p}).\]
Here ${\bf p}=\{1,2,\dots ,n\}\in C_n(\mathbb{C})$.
\end{theorem}

The isomorphism above is defined as follows: a braid can be thought of as a path in the configuration space $\hat{C}_n(\mathbb{C})$. However, seen as a path in this configuration space it is not closed, since the endpoint of a strand may be different from its starting point. By composing with the quotient map to the unordered configuration space, this path becomes a closed path, and therefore defines an element of the fundamental group.

\subsection{Mapping class groups}

The braid groups $B_n$ can also be realized as mapping class groups, which exposes the very interesting relation
with the topology of surfaces.
Let $M,N$ be compact manifolds, possibly with boundary. We denote by $C(M,N)$ the space of all continuous functions from $M$ to $N$. There is a natural topology on $C(M,N)$,  the compact open topology, defined as follows: for any $K \subset M$ compact and $U \subset N$ open,  set:
\[V(K,U):= \{ f \in C(M, N): f(K)\subset U\}.\]
The compact open topology is the smallest topology for which all $V(K,U)$ are open. 
We denote by $\mathsf{\mathsf{Homeo}}(M)$ the group of homeomorphisms of $M$. This is a subset of $C(M,M)$ and in this way inherits 
a topology which gives it the structure of a topological group.

Let us now specialize the discussion to the case of an orientable compact surface $S$. The group of orientation preserving homeomorphism of $S$  that fix the boundary pointwise will be denoted by $\mathsf{Homeo}^+(S, \partial S)$. The connected component of the identity in 
$\mathsf{Homeo}^+(S, \partial S)$ is denoted by $\mathsf{Homeo}_0(S,\partial S)$ and consists of the set of homeomorphism that are isotopic to the identity via an isotopy that fixes the boundary pointwise.

\begin{definition}
The mapping class group of a surface $S$ denoted $\mathsf{Mod}(S)$ is defined as:
\[\mathsf{Mod}(S)= \mathsf{Homeo}^+(S,\partial S)/ \mathsf{Homeo}_0(S, \partial S).\]
\end{definition}

\begin{remark}
The mapping class group can alternatively be defined by replacing homeomorphisms by diffeomorphism and 
isotopies by homotopies. The following is a list of possible definitions of $\mathsf{Mod}(S)$:

\begin{eqnarray*}
\mathsf{Mod}(S)&\cong & \mathsf{Homeo}^+(S,\partial S)/ \mathsf{Homeo}_0(S, \partial S)\\
& \cong &  \mathsf{Diff}^+(S,\partial S)/ \mathsf{Diff}_0(S, \partial S)\\
& \cong & \pi_0(\mathsf{Homeo}^+(S,\partial S))\\
& \cong & \pi_0(\mathsf{Diff}^+(S,\partial S)).\\
\end{eqnarray*}
\end{remark}

\begin{example}
The notation $\mathsf{Mod}(S)$ for the mapping class group is motivated by the example of the torus $\mathbb{T}$, where:
\[\mathsf{Mod}(\mathbb{T})\cong SL_2(\mathbb{Z}).\]
The isomorphism is defined by the action of $\mathsf{Mod}(\mathbb{T})$ on homology.
\end{example}

\begin{definition}
Let $S$ be a surface and $Q\subset S$ a finite set of marked points. We denote by $\mathsf{Homeo(S,Q)}$ the group of homeomorphisms of $S$ that fix $Q$ as a set and fix the boundary pointwise.
The mapping class group of $S$, seen as a surface
with marked points is
\[\mathsf{Mod}(S,Q):=\pi_0( \mathsf{Homeo}^+(S, \partial S)\cap \mathsf{Homeo}(S,Q)).\]

\end{definition}

\begin{example}[Alexander-Tietze Theorem]\label{AT}
The mapping class group of the disk is trivial:
\[\mathsf{Mod}(D)=\{0\}.\]
In order to prove this fact we consider an automorphism $f: D \rightarrow D$ that fixes the boundary.
Then we define the isotopy:
\[  
h_t(z)=\begin{cases}
z &\textit{ if }t\leq |z|\leq 1, \\
t h(\frac{z}{t})& \textit{ if } |z| \leq t. 
\end{cases}
\]

\end{example}

\begin{example}

An example of a nontrivial mapping class is that of the half twist. Here $S$ is the disk with two marked points, that we may take to be $p=(-\frac{1}{2},0)$ and $q=(\frac{1}{2},0)$. The half twist $H$ is the class of the diffeormorphism
described by the following picture:

    \begin{figure}[h!]
\center
\includegraphics[width=100mm]{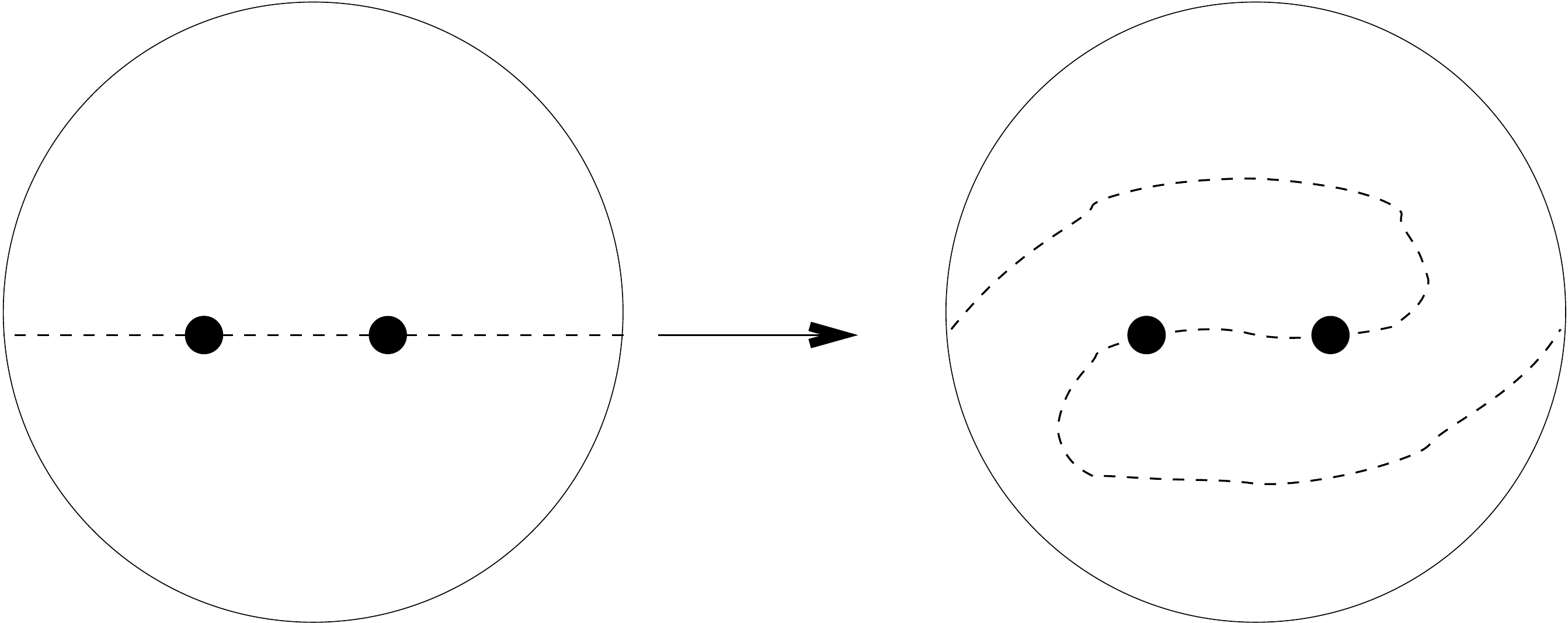}]
\caption{Halftwist of a two pointed disk}
\end{figure}
    
\end{example}

\begin{theorem}
Let $D$ be a disk with $n$ marked points $Q=\{p_1, \dots, p_n\}$. There is an isomorphism:

\[\psi: B_n \rightarrow \mathsf{Mod}(D,Q),\]
characterized by the property that:
\[\sigma_i \mapsto H_i.\]

Here $H_i$ denotes the class of the homeomorphism which is supported on a disk containing only the $i$-th and
$i+1$-th punctures, and which acts as a half twist on that disk.
\end{theorem}
The reader interested in the proof of this theorem may consult \cite{BB}.
In this section we have closely followed the book by Farb and Margalit \cite{Farb}, which is an excellent reference for learning about mapping class groups.
\subsection{Pure braids}

There is a natural group homomorphism $ \pi: B_n \rightarrow \Sigma_n$ which sends a braid to the corresponding permutation of the marked points. Alternatively, the homomorphism is characterized by the property that:
\[\sigma_i \mapsto (i, i+1).\]

\begin{definition}
The pure braid group on $n$ strands, denoted by $P_n$, is the kernel of the homomorphism $\pi: B_n \rightarrow \Sigma_n$.

\end{definition}

The group $P_n$ consists of those braids such that each strand starts and ends at the same point. These are closed paths in the configuration space $\hat{C}_n(\mathbb{C})$, and in this way one obtains an identification:
\[P_n \cong \pi_1(\hat{C}_n(\mathbb{C}),{\bf p}).\]

\begin{remark}
Let $G$ be an abelian group. An Eilenberg-Maclane space  $K(G, n)$ is a topological space such that: 
\begin{itemize}
\item $\pi_n(K(G,n))\cong G.$
\item $\pi_m(K(G,n))\cong \{0\}$ for $m\neq n$.
\end{itemize}
Given $G$ and $n$, the spaces $K(G,n)$ exist and are well defined up to homotopy equivalence.
The configuration spaces of points in the plane are important examples of Eilenberg-Maclane spaces:
\[C_n(\mathbb{C})\cong K(B_n, 1).\]
\[\hat{C}_n(\mathbb{C}) \cong K(P_n, 1).\]

\end{remark}

\section{Homological representations }\label{section3}

In this section we discuss the Burau representations of braid groups, their geometric interpretation as action
of mapping classes on homology, and their relation with the Alexander-Conway polynomial of links. We also introduce 
the Lawrence-Krammer-Bigelow representation, which has the remarkable property of being faithful.
Our exposition follows closely the book by Kassel and Turaev \cite{KT}.
\subsection{The Burau representation}

Let us denote by $\Lambda$ the ring of Laurent polynomials with integer coefficients:
\[\Lambda:= \mathbb{Z}[t,t^{-1}].\]

In \cite{Burau}, W. Burau constructed representations of the braid group $B_n$ on the space $\Aut(\Lambda^n)$ of $n\times n$ matrices with coefficients in 
$\Lambda$. Let us define the $n\times n$ matrix $U_i$ as follows:

\begin{equation*}
U_i= \left( \begin{array}{cccc} I_{i-1} &0&0&0\\
0&1-t& t&0\\
0&1& 0&0\\
0&0&0&I_{n-i-1}\\
\end{array}\right)
\end{equation*}

\begin{proposition}
There is a unique group homomorphism $\psi_n: B_n \rightarrow \Aut(\Lambda^n)$ characterized by the property that:
\[\sigma_i \mapsto U_i.\] 
This homomorphism is called the Burau representation of $B_n$.
\end{proposition}
\begin{proof}
We begin by proving that the matrices $U_i$ are invertible. Let us set 
\begin{equation*}
U= \left( \begin{array}{cc}
1-t&t \\
1&0\\
\end{array}\right).
\end{equation*}
Since the matrices $U_i$ are block diagonal matrices whose blocks are either identity matrices of $U$ it is enough 
to prove that $U$ is an invertible matrix. For this we exhibit an explicit inverse:

\begin{equation*}
U^{-1}= \left( \begin{array}{cc}
0&1 \\
t^{-1}&1-t^{-1}\\
\end{array}\right).
\end{equation*}
One can directly check that $U U^{-1}=U^{-1}U =\id$.
The block form of the matrices implies that $U_i U_j=U_jU_i$ provided that $|i-j|\geq 2$. Finally, it remains to show
that \[U_i U_{i+1}U_i=U_{i+1}U_iU_{i+1}.\] This can be checked in the case $n=3$ and becomes the following exercise in matrix multiplication:
\begin{eqnarray*}
 \left( \begin{array}{ccc}
1-t&t&0 \\
1&0&0\\
0&0&1
\end{array}\right)
 \left( \begin{array}{ccc}
1&0&0 \\
0&1-t&t\\
0&1&0
\end{array}\right)
 \left( \begin{array}{ccc}
1-t&t&0 \\
1&0&0\\
0&0&1
\end{array}\right)=\\
 \left( \begin{array}{ccc}
1&0&0 \\
0&1-t&t\\
0&1&0
\end{array}\right)
 \left( \begin{array}{ccc}
1-t&t&0 \\
1&0&0\\
0&0&1
\end{array}\right)
 \left( \begin{array}{ccc}
1&0&0 \\
0&1-t&t\\
0&1&0
\end{array}\right).
\end{eqnarray*}

\end{proof}

\begin{remark}
Setting $t=1$, each of the matrices $U_i$ becomes a permutation matrix, and one obtains representations
of $B_n$ that factor through the symmetric groups. For this reason, the Burau representations may be seen as 
deformation of the usual permutation representation.
\end{remark}

The Burau representation admits a one dimensional invariant subspace. By taking the quotient of $\Lambda^n$
by this invariant subspace one obtains the reduced Burau representations.

\begin{proposition}
Let $n$ be a natural number greater that $2$ and $V_1,\dots,V_{n-1}$ be the $(n-1) \times (n-1)$ matrices
defined as follows:
\begin{equation*}
V_1= \left( \begin{array}{ccc}
-t&0&0 \\
1&1&0\\
0&0& I_{n-3}
\end{array}\right), \,\,\,
V_{n-1}=\left( \begin{array}{ccc}
I_{n-3}&0&0 \\
0&1&t\\
0&0& -t
\end{array}\right),
\end{equation*}
and for $ 1 <i< n-1:$

\begin{equation*}
V_i= \left( \begin{array}{ccccc}
I_{i-2}&0&0&0&0 \\
0&1&t&0&0\\
0&0&-t&0&0\\
0&0&1&1&0\\
0&0&0&0&I_{n-i-2}
\end{array}\right).
\end{equation*}

Also, Let $C$ be the $n\times n$ matrix over $\Lambda$:
\begin{equation*}
C= \left( \begin{array}{ccccccc}
1&1&1&\dots&1 \\
0&1&1&\dots&1\\
0&0&1&\dots &1\\
\vdots&\vdots&\vdots &\ddots&\vdots \\
0&0&0&\dots &1
\end{array}\right).
\end{equation*}

Then:
\begin{equation*}
C^{-1}U_i C= \left( \begin{array}{cc}
V_i &0 \\
X_i&1
\end{array}\right),
\end{equation*}
where $X_i$ is the row of length $n-1$ which is $(0,\dots ,0)$ in case $i \neq n-1$ and is
$(0, \dots 0,1)$ for $i=n-1$. 
\end{proposition}
\begin{proof}
We set 
\begin{equation*}
W_i:= \left( \begin{array}{cc}
V_i &0 \\
X_i&1
\end{array}\right).
\end{equation*}
It suffices to prove that $U_i C= C W_i$ for all $i=1, \dots , n-1$.
A direct computation shows that:
\begin{equation*}
U_iC= \left( \begin{array}{ccccccc}
1&1&1&\dots&1 &1\\
0&1&1&\dots&1&1\\
0&0&1&\dots &1&1\\
\vdots&\vdots&\vdots &\ddots&\vdots & \vdots\\
0&0&0&1-t& \dots&1\\
0&0&0&1& \dots&1\\
\vdots&\vdots&\vdots &\ddots&\vdots& \vdots \\
0&0&0&\dots &0 &1
\end{array}\right).
\end{equation*}
Similarly, a simple calculation gives:
\begin{equation*}
CW_i= \left( \begin{array}{ccccccc}
1&1&1&\dots&1 &1\\
0&1&1&\dots&1&1\\
0&0&1&\dots &1&1\\
\vdots&\vdots&\vdots &\ddots&\vdots & \vdots\\
0&0&0&1-t& \dots&1\\
0&0&0&1& \dots&1\\
\vdots&\vdots&\vdots &\ddots&\vdots& \vdots \\
0&0&0&\dots &0 &1
\end{array}\right).
\end{equation*}
Here, the matrices on the right hand side of the equations are those obtained from $C$ by replacing the entry $(i,i)$ by
$1-t$ and the entry $(i+1,i)$ by $1$. This completes the proof.
\end{proof}
Since conjugation with the matrix $C$ is an automorphism, we conclude that the matrices $W_i$ satisfy the braid relations. Moreover, since $\det(W_i)=\det(V_i)$, the matrices $V_i$ are invertible as elements of $\End(\Lambda^{n-1})$. The fact that the last column of the matrices $W_i$ is nonzero only in the last entry also implies that the matrices $V_i$ satisfy the braid relations. These observations allow us to make the following definition:

\begin{definition}
For $n \geq 2$ the reduced Burau representation of the braid group $B_n$ is the representation ${\psi^r_n}: B_n \rightarrow \Aut(\Lambda^{n-1})$ characterized by the property that:
\[\sigma_i \mapsto V_i.\]
For $n=2$ the reduced Burau representation is given by $\psi^r_2(\sigma_1)=-t.$
\end{definition}

\begin{remark}
The Burau representation is known to be faithful for $ n\leq 3$  and not faithful for $n \geq 5$. It is not known whether it is faithful for $n=4$. A theorem of Bigelow asserts that $\psi_4$ is faithful if and only if the Jones polynomial detects the unknot.
\end{remark}

\subsection{Homological interpretation}

Let us denote by $D$ the closed disk in the plane with $n$ distinguished interior points $ p_1< \dots < p_n \in D \cap \mathbb{R}$ and set $Q:=\{ p_1,\dots ,p_n\}$. Observe that for any point $p$ in the interior of $D$:
\[ H_1(D-\{p\},\mathbb{Z})\cong \mathbb{Z},\]
is generated by the homology class of a small circle around $p$ oriented counterclockwise. Set $\Sigma:=D-Q$ and fix a basepoint $d \in \partial D$. We define the group homomorphism $\phi: \pi_1(\Sigma, d)\rightarrow \mathbb{Z}$ by 
\[ [\gamma] \mapsto \sum_{i=1}^n w_i(\gamma),\]
where $w_i(\gamma)$ is the winding number of $\gamma$ around $p_i$.
The kernel of $\phi$ determines a covering space $\tilde{\Sigma} \rightarrow \Sigma$, whose group of covering transformation is the infinite cyclic group $\mathbb{Z}$. We choose a point $\tilde{d}\in \tilde{\Sigma}$ over $d$ and set:

\[\tilde{H}:= H(\tilde{\Sigma}, \mathbb{Z} \tilde{d},\mathbb{Z}).\] 

Let $F$ be a self homeomorphism of $D$ that permutes the elements of $Q$. The restriction $f$ of $F$ to $\Sigma$ preserves the total winding number i.e.
\[ \phi \circ f_*= \phi ,\]
and therefore $f$ can be lifted to a homeomorphism $\tilde{F}$ of $\tilde{\Sigma}$. This construction defines a representation:

\[\Psi_n: \mathsf{Mod}(D,Q) \rightarrow \Aut(\tilde{H}).\] 

Since the group of deck transformations of $\tilde{\Sigma} \rightarrow \Sigma$ is the infinite cyclic group, the homology
group $\tilde{H}$ has an action of $\mathbb{Z}$, and therefore it is a module over the ring $\Lambda$.
The following theorem shows that the homological representation constructed above coincides with the Burau representation.
\begin{theorem}
There exists a an isomorphism of $\Lambda$-modules $\mu: \Lambda^n \rightarrow \tilde{H}$ such that the 
following diagram commutes:
\[
\xymatrix{
B_n \ar[r]^\cong \ar[d]_{\psi_n} & \mathsf{Mod}(D,Q)\ar[d]^{\Psi_n}\\
\Aut(\Lambda^n) \ar[r]^{\underline{\mu}} & \Aut(\tilde{H}), }
\]
Here $\underline{\mu}$ is the isomorphism induced by $\mu$.
\end{theorem}
\begin{proof}
For a proof the reader may consult  \cite{KT} Theorem 3.7.
\end{proof}

\subsection{Braids, knots and the Alexander Conway polynomial}

A knot is an isotopy class of an embedding of a circle in three dimensional euclidian space. More generally, a link with $k$ components is an isotopy class of an embedding of $k$ circles in $\mathbb{R}^3$. The fundamental problem of knot theory is the classification of knots and links. 
The following is a typical example of a knot:

\[ \xy 
(6,9)*{}="1"; 
(-8.5,-1)*{}="2"; 
"1";"2" **\crv{~*=<.5pt>{.} (0,30)}?(.75);
(-6.5,8)*{}="1"; 
(-.5,-9)*{}="2"; 
"1";"2" **\crv{~*=<.5pt>{.} (-28.5,9.3)}?(.7);
(-9.5,-3.35)*{}="1"; 
(8.5,-3)*{}="2"; 
"1";"2" **\crv{~*=<.5pt>{.} (-17.67,-24.19)}?(.7);
(1,-10)*{}="1"; 
(6.5,7.13)*{}="2"; 
"1";"2" **\crv{~*=<.5pt>{.} (17.67,-24.19)}?(.7);
(11,-1)*{}="1"; 
(-4,8)*{}="2"; 
"1";"2" **\crv{~*=<.5pt>{.} (28.5,9.3)}?(.93);
\endxy \]

The relationship between links and braids is given by the operation of closure of a braid,
which produces a link by connecting the strands of the braid, as explained in the following diagram:

\[ \xy 
(0,0)*+{\includegraphics{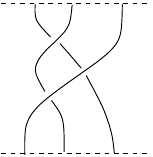}}; 
(0,5)*{}; 
\endxy 
\quad 
\mapsto 
\quad 
\xy 
(0,0)*+{\includegraphics{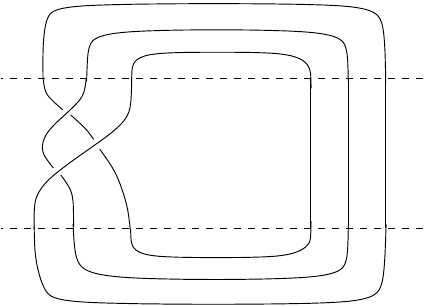}}; 
\endxy \]

A link obtained as the closure of a braid comes with a natural orientation, given by declaring the strands
of the braid to be flowing downwards.
\begin{theorem}[Alexander \cite{Alexander}]
Any oriented link can be obtained as the closure of a braid.
\end{theorem}
\begin{proof} For a proof of this theorem the reader may consult \cite{KT} Theorem 2.3.
\end{proof}
The theorem above immediately raises the question of deciding when two braids produce
the same oriented link upon closure. This question is answered by Markov's theorem.
We say that two braids $\beta ,\beta' \in B_n$ are related by the Markov move $M_1$ if their are conjugate, i.e. 
if there exists a $\gamma \in B_n $ such that $ \gamma \beta \gamma^{-1}=\beta'$.
The following picture describes the action of the operation $M_1$.

\[
\qquad
\xy 
(0,0)*+{\includegraphics{R3.pdf}}; 
\endxy 
\quad 
\mapsto 
\quad 
\xy 
(0,0)*+{\includegraphics{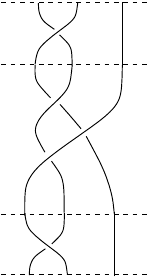}}; 
\endxy \]

Let $\iota_n: B_n \rightarrow B_{n+1}$ be the natural inclusion of groups characterized by $ \iota(\sigma_i)=\sigma_i$.
Two braids $\beta \in B_n, \beta'\in B_{n+1}$ are related by the second Markov move  $M_2$ if $ \sigma_n^{\pm} \iota(\beta)= \beta'$. The following picture describes the action of the second Markov move on a braid:

\[ 
\qquad 
\xy 
(0,0)*+{\includegraphics{R3.pdf}}; 
\endxy 
\quad 
\mapsto 
\quad 
\xy 
(0,0)*+{\includegraphics{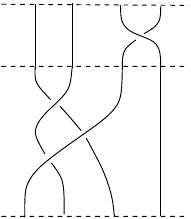}}; 
\endxy
\]
\begin{definition}
The Markov equivalence relation $M$ is the smallest equivalence relation on the set of all braids such that
two braids are equivalent if they are related by the Markov moves $M_1$ and  $M_2$. We say that two braids are Markov equivalent if they are equivalent with respect to $M$.
\end{definition}

\begin{theorem}[Markov \cite{Markov}]
Two braids are Markov equivalent if and only if their closures represent the same oriented link.
\end{theorem}
\begin{proof}
The interested reader may consult \cite{KT} Theorem 2.8.
\end{proof}

\begin{definition}
A Markov function $f$ with values in a set $X$ is a sequence of functions $f_n: B_n \rightarrow X$ with the property that:

\begin{itemize}
\item $f_n(\gamma^{-1} \beta \gamma)= f_n ( \beta)$.
\item $f_n(\beta)=f_{n+1}(\sigma_n \iota_n(\beta))=f_{n+1}(\sigma_n^{-1} \iota_n(\beta))$.
\end{itemize}
\end{definition}

Clearly, Markov's theorem implies that any Markov function produces a link invariant with values in the set $X$. Let us now describe a Markov function provided by the Burau representation.
Let $g: \Lambda= \mathbb{Z}[t,t^{-1}]\rightarrow \mathbb{Z}[s,s^{-1}]$ be the ring homomorphism characterized by 
$t \mapsto s^2$. We define the group homomorphsim \[\langle \,\, \rangle: B_n \rightarrow \mathbb{Z},\]  by setting
$\langle \sigma_i \rangle =1$.

 For $n\geq 2$ we define the function $f_n: B_n \rightarrow X=\mathbb{Z}[s,s^{-1}]$ by the formula:

\[ f_n(\beta)=(-1)^{n+1} \frac{s^{- \langle \beta \rangle}(s-s^{-1})}{s^n-s^{-n}}g  \left(\det \left(\psi^r_n(\beta)-\id \right)\right).\]

\begin{theorem}
The sequence of functions $f_n$ defined above is a Markov function with values in the set $X=\mathbb{Z}[s,s^{-1}]$. 
\end{theorem}

\begin{proof}
The interested reader may consult \cite{KT} Lemma 3.12.
\end{proof}

The Markov function $f_n$ defined above in terms of the Burau representation defines an invariant of oriented links known as the Alexander-Conway polynomial.

\begin{definition}
The Alexander-Conway polynomial of an oriented link $L$, denoted $\nabla(L)$, is the polynomial
$f_n(\beta) \in \mathbb{Z}[s,s^{-1}]$, where $\beta$ is any braid whose closure is $L$.
\end{definition}

For the purpose of computation it is often useful to describe the skein relations for the Alexander Conway polynomials.
These relations express the change in the value of the polynomial that occurs when modifying the crossing of a link diagram. A Conway triple $L_+, L_-, L_0$ is a triple of link diagrams that differ only locally at one crossing, which looks, respectively, as follows: 

\[
\xy 
(6,6)*{}="tl"; 
(-6,6)*{}="tr"; 
(6,-6)*{}="bl"; 
(-6,-6)*{}="br"; 
{\ar|{\hole \; } "bl";"tr"}; 
{\ar "br";"tl"}; 
(0,-10)*{L_{+}}; 
\endxy 
\qquad \qquad 
\xy 
(-6,6)*{}="tl"; 
(6,6)*{}="tr"; 
(-6,-6)*{}="bl"; 
(6,-6)*{}="br"; 
{\ar|{\hole \; } "bl";"tr"}; 
{\ar "br";"tl"}; 
(0,-10)*{L_{-}}; 
\endxy 
\qquad \qquad 
\xy 
(-6,6)*{}="tl"; 
(6,6)*{}="tr"; 
(-6,-6)*{}="bl"; 
(6,-6)*{}="br"; 
"bl";"tl" **\crv{(-1,0)}?(1)*\dir{>}; 
"br";"tr" **\crv{(1,0)}?(1)*\dir{>}; 
(0,-10)*{L_{0}}; 
\endxy 
\]

\begin{theorem}
The Alexander-Conway polynomial is the unique invariant of oriented links $\nabla$ with values in $\mathbb{Z}[s,s^{-1}]$ which satisfies the following properties:
\begin{itemize}
\item $\nabla(L)=1$ if $L$ is the unknot. 
\item For any Conway triple $L_+,L_-,L_0$:
\[\nabla(L_+)-\nabla(L_-)=(s^{-1}-s)\nabla(L_0).\]
\end{itemize}
\end{theorem}
\begin{proof}
The proof can be found in \cite{KT} Theorem 3.13.
\end{proof}

It is fairly easy to show that there is at most one invariant satisfying the skein relations above. The existence is the more interesting part of the theorem, and that is where the explicit construction of the invariants, provided by the Burau representations, is needed.

\subsection{The Lawrence-Krammer-Bigelow representation}

We will now discuss another homological representation introduced by R. Lawrence \cite{Lawrence} and studied by D. Krammer \cite{Krammer} and S. Bigelow \cite{Bigelow}. Let us fix a natural number $n\geq 1$ and denote by $D$ the unit disk with $n$ distinguished points $Q=\{ p_1, \dots, p_n\}$. We also set \[ \Sigma:= D-Q, \,\,\mathcal{F}:= \hat{C}_2(\Sigma), \,\,\mathcal{C}:=C_2(\Sigma) .\]

Here we use the notation introduced above for configuration spaces. That is, $\mathcal{F}$ is the configuration space of two ordered points in $\Sigma$, and $\mathcal{C}$ is the configuration space of two unordered points in $\Sigma$.
There is a natural double sheeted covering $\mathcal{F} \rightarrow \mathcal{C}$.
A point in $\mathcal{C}$ is an unordered set of distinct points of $\Sigma$ and we use the notation $\{x,y\}$ for points of $\mathcal{C}$. A path in $\mathcal{C}$ can be represented as a pair of paths $\{\zeta_1, \zeta_2\}$ where $\zeta_i: I \rightarrow \Sigma$ such that for all $s\in I, \zeta_1(s)\neq \zeta_2(s)$. A loop in $\mathcal{C}$ is a path such that 
$\{\zeta_1(0), \zeta_2(0)\}=\{ \zeta_1(1), \zeta_2(1)\}$. That is, a pair of maps $(\zeta_1, \zeta_2)$ defines a loop in $\mathcal{C}$ if either they define a loop in $\mathcal{F}$ or define a path in $\mathcal{F}$ that permutes the starting points. 

Given a loop $\zeta=\{\zeta_1, \zeta_2\} $ in $\mathcal{C}$ we define the number $ w(\zeta)\in \mathbb{Z}$ as follows. If the pair $(\zeta_1, \zeta_2)$ defines a loop in $\mathcal{F}$, then each of the $\zeta_i$ are loops in $\Sigma$ and we set \[w(\zeta):= \phi(\zeta_1)+ \phi(\zeta_2),\]
where, as above, $\phi$ is the group homomorphism that counts the total winding number with respect to $Q$.
If the pair $(\zeta_1,\zeta_2)$ does not define a loop in $\mathcal{F}$, then it defines a path that permutes the endpoints. Therefore $ \zeta_1 \circ \zeta_2$ defines a loop is $\Sigma$ and we set:
\[w(\zeta):= \phi(\zeta_1 \circ \zeta_2).\]

We define a second numerical invariant $u(\zeta)$ as follows. Consider the map $\gamma_\zeta: I \rightarrow S^1$ given by the formula:

\[ \gamma_\zeta (s) :=\left( \frac{\zeta_1(s)- \zeta_2(s)}{|\zeta_1(s)-\zeta_2(s)|} \right)^2.\]
Observe that this map is well defined for $\zeta$, since $\zeta_1$ and $\zeta_2$ play a symmetric role. Since $\zeta$ is a loop, we also conclude that $\gamma_\zeta(0)= \gamma_\zeta(1)$ and therefore $\gamma_\zeta$ defines a map
$S^1 \rightarrow S^1$ whose degree is the value of $u$ at $\zeta$. It is a simple exercise to prove that the maps $w$ and $u$ are well defined on homotopy classes and are group homomorphisms $\pi_1(\mathcal{C})\rightarrow \mathbb{Z}$.

We fix two distinct points $d_1,d_2 \in \partial(\Sigma)$ and take $\{d_1,d_2\}$ as the basepoint of $\mathcal{C}$.
We denote by \[\rho: \pi_1(\mathcal{C}, \{d_1,d_2\})\rightarrow \mathbb{Z}\oplus \mathbb{Z},\]
the group homomorphism given by the formula:
\[\rho(\zeta)= q^{w(\zeta)}t^{ u(\zeta)},\]
where $q$ and $t$ are the canonical generators of $ \mathbb{Z} \oplus \mathbb{Z}$.
\begin{lemma}
The group homomorphism $\rho: \pi_1(\mathcal{C}, \{d_1,d_2\})\rightarrow \mathbb{Z}\oplus \mathbb{Z}$ is surjective.
\end{lemma}
\begin{proof}
It suffices to prove that both $q$ and $t$ are in the image of $\rho$.
Let us first prove that $q$ is in the image. Let $\zeta=\{\zeta_1,\zeta_2\}$ be the loop in $\mathcal{C}$ such that $\zeta_1$ is constant and $\zeta_2$ is a small counterclockwise loop around $p_1$. Then:
\[w(\zeta)=\phi(\zeta_1)+\phi(\zeta_2)=1.\]
Since $\zeta_2$ stays close to $p_1$ and $\zeta_1$ is constant, we conclude that $u(\zeta)=0$. Thus $\rho(\zeta)=q$.

Let us now prove that $t$ is in the image of $\rho$. Consider  a  small disk $B\subset \Sigma$ and define the loop $\varphi=\{\varphi_1,\varphi_2\}$ as follows. Fix two distinct points $x,y \in \partial(B)$. Set $\varphi_1$ the path that goes along $\partial B$ from $x$ to $y$ in counterclockwise direction and  $\varphi_2$ the path that goes along $\partial B$ from $y$ to $x$ in counterclockwise direction. Then:
\[w(\varphi)= \phi(\varphi_1 \circ \varphi_2)=0,\]
and \[u(\varphi)= \mathsf{deg}(\id)=1.\]
We conclude that $\rho(\varphi)=t.$ 
\end{proof}

Let $\tilde{\mathcal{C}} \rightarrow \mathcal{C}$ be the covering space that corresponds to the subgroup $\ker(\rho)$ of $\pi_1(C,\{d_1,d_2\})$. The group $\mathbb{Z} \oplus \mathbb{Z}$ acts as the group of deck transformations of $\tilde{\mathcal{C}}$. Therefore, the space
\[\mathcal{H} := H_2(\tilde{\mathcal{C}},\mathbb{Z}),\]
has the structure of a module over the ring:

\[ R:= \mathbb{Z}[q,q^{-1},t, t^{-1}].\]

We will now describe an action of the braid group $B_n$ on the $R$-module $\mathcal{H}$. We will again use the identification
\[B_n \cong \mathsf{Mod}(D,Q).\]

Take a homeomorphism $f$ of $D$ that fixes the boundary pointwise and permutes the elements of $Q$. The map $f$ induces a homeomorphism $\hat{f}: \mathcal{C}\rightarrow \mathcal{C}$ given by:
\[\hat{f}(\{x,y\}):= \{f(x),f(y)\}.\]

Since $f$ fixes the boundary pointwise, we know that $\hat{f}(\{d_1,d_2\})=\{d_1,d_2\}$ and therefore $\hat{f}$ induces an isomorphism $\hat{f}_*: \pi_1(\mathcal{C}, \{d_1,d_2\})\rightarrow \pi_1(\mathcal{C}, \{d_1,d_2\})$.

\begin{lemma}
The homomorphism $\hat{f}_*$ preserves the invariant $\rho$, i.e.
$\rho \circ \hat{f}_*= \rho.$
\end{lemma}
\begin{proof}
We need to prove that $w \circ \hat{f}_*= w$ and $u \circ \hat{f}_*= u$. 
The first statement reduces to the fact that the total winding number is preserved by $\hat{f}_*$, which we have already seen.
Let us now prove the second statement. Consider the inclusion of configuration spaces:
\[\iota: \mathcal{C} =C_2(\Sigma) \hookrightarrow C_2(D),\]
induced by the inclusion $ \Sigma \hookrightarrow D$. The definition of the invariant $u$ can be extended word by word to an invariant $\hat{u}$ of loops in $C_2(D)$ in such a way that the diagram:

\[ \xymatrix{
\pi_1(\mathcal{C},\{x,y\})\ar[r]^\iota \ar[dr]^u & \pi_1(C_2(D),\{x,y\}) \ar[d]^{\hat{u}}\\
&\mathbb{Z}
}
\]
commutes. On the other hand, $\hat{f}$ extends to a homeomorphism $\overline{f}$ of $C_2(D)$ which, by the 
Alexander-Tietze theorem proved in Example \ref{AT}, is isotopic to the identity. Thus $\overline{f}$ acts trivially on the fundamental group of $C_2(D)$. Then, for any element $\gamma \in \pi_1(\mathcal{C},\{x,y\})$:

\[ u \circ \hat{f}_*(\gamma)= \hat{u} \circ \iota \circ \hat{f}_* (\gamma)=\hat{u} \circ \overline{f}_*\circ  \iota \circ  \gamma =\hat{u} \circ \iota (\gamma)= u(\gamma).\]
\end{proof}

The previous lemma implies that the homeomorphism $\hat{f}$ lifts naturally to a homeomorphism $ \tilde{f}: \tilde{\mathcal{C}} \rightarrow \tilde{\mathcal{C}}$ that commutes with the deck transformations of the covering.
This homeomorphism induces a homomorphism in homology $\tilde{f}_*: H_2(\tilde{\mathcal{C}},\mathbb{Z})\rightarrow H_2(\tilde{\mathcal{C}},\mathbb{Z})$.

\begin{definition}
The Lawrence-Krammer-Bigelow representation of the braid group $B_n$ is the homomorphism:
\[B_n \rightarrow \Aut_R(H_2(\tilde{\mathcal{C}}, \mathbb{Z})),\]
given by the formula:
\[f \mapsto \tilde{f}_*.\]
\end{definition}

\begin{theorem}[Krammer \cite{Krammer}, Bigelow \cite{Bigelow}]
The following statements hold for all $n \geq 2$:
\begin{enumerate}
\item There is a natural isomorphism of $R$-modules:

\[\mathcal{H}\cong R^{\frac{n(n-1)}{2}}.\]
\item The Lawrence-Krammer-Bigelow representation $B_n \rightarrow \Aut_R(\mathcal{H})$ is faithful.
\end{enumerate}
\end{theorem}
\begin{proof}
The interested reader may consult \cite{KT} Theorem 3.15.
\end{proof}

The second statement of the theorem implies that $B_n$ is a linear group, i.e. it is isomorphic to a group of matrices with real coefficients. This can be obtained by setting $q,t$ to be algebraically independent real numbers.   

\section{The Knizhnik-Zamolodchikov connection}\label{section4}
We will now describe representations of braid groups that arise by monodromy of certain flat connections, known as Knizhnik-Zamolodchivov connections,  on configuration spaces of points in the plane. More about this connections and the  relationship with the Jones polynomial can be found in Kohno's book \cite{Kohno}.
\subsection{Cohomology of configuration spaces and the KZ connection}

As before, we denote by $\hat{C}_n(\mathbb{C})$ the configuration space of $n$ ordered points in the plane.
The cohomology of $\hat{C}_n(\mathbb{C})$ has been computed explicitly by Arnold in \cite{Arnold}. Since the original article is in russian, some readers may prefer to read \cite{Sinha} for a detailed explanation of the computation.
Let us now describe Arnold's result.
\begin{definition}
The algebra $\mathsf{A}_n$ is the graded commutative algebra over $\mathbb{C}$ generated by degree one 
elements $a_{ij}=a_{ji}$ for $i\neq j, 1 \leq i ,j \leq n$ modulo the Arnold relation:

\[a_{ij} a_{jk}+ a_{ki}a_{ij}+ a_{jk} a_{ki}=0, \textit{ for } i<j<k.\]
\end{definition}
 We also define the differential forms $w_{ij}$ for $1 \leq i,j \leq n$  in $\hat{C}_n(\mathbb{C})$ given by:
 \[w_{ij}=d (\log(z_i -z_j))= \frac{dz_i -dz_j}{z_i-z_j}.\]
 
 \begin{theorem}[Arnold]
There is a homomorphism of differential graded algebras $ A_n \rightarrow \Omega(\hat{C}_n(\mathbb{C}))$ given by:
\[a_{ij} \mapsto w_{ij}.\]
Moreover, this homomorphism induces an isomorphism in cohomology
\[A_n \cong H(\hat{C}_n(\mathbb{C}), \mathbb{C}).\]
 \end{theorem}
\begin{proof}
The fact that the formula above gives a homomorphism is an explicit computation with the differential forms $w_{ij}$. For the proof that the map is an isomorphism we recommend \cite{Sinha}.
\end{proof}

Arnold's computation can be used to define natural flat connections on configuration spaces.
\begin{definition}
For each $n \geq 2$, the Kohno-Drinfeld Lie algebra is the Lie algebra $\mathfrak{t_n}$ generated by the symbols
$t_{ij}=t_{ji}$  for $1\leq i, j \leq n$ modulo the relations:

\[
[t_{ij}, t_{kl}]= 0 \textit{ if } \#\{i,j,k,l\}=4,\]

\[[t_{ij}, t_{ik}+t_{jk}]=0  \textit{ if } \# \{i,j,k\}=3.\]

\end{definition}

\begin{definition}
The Knizhnik-Zamolodchikov connection is the connection on the configuration space $\hat{C}_n(\mathbb{C})$ with values in the Kohno-Drinfeld Lie algebra, given by the formula:
\[\theta_n:= \sum_{i<j} t_{ij}w_{ij}. \]
\end{definition}

\begin{lemma}
The  Knizhnik-Zamolodchikov connection $\theta_n$ is flat, i.e.
\[d\theta_n+ \frac{1}{2} [\theta_n,\theta_n]=0.\]
\end{lemma}
\begin{proof}
Since each of the forms $w_{ij}=d(\log(z_i-z_j))$ is closed, it is enough to prove that:
\[ [\theta_n,\theta_n]=0.\]
For this we use the Arnold relations as follows:
\[[\theta_n,\theta_n]=\sum_{i<j, k<l} [t_{ij} \omega_{ij}, t_{kl} w_{kl}]=\sum_{i<j, k<l}[t_{ij},t_{kl}] w_{ij}w_{kl}.\]
In view of the first relation in the Drinfeld-Kohno Lie algebra, it suffices to sum over sets of indices so that
$\#\{i,j,k,l\}=3$. Thus, the expression above is equal to:
\[\sum_{i<j, k<l, \\\#\{i,j,k,l\}=3}[t_{ij},t_{kl}] w_{ij}w_{kl}.\]
We now rewrite this expression by considering all the possible ways in which a pair of indices can be equal, and obtain:
\begin{eqnarray*}
\sum_{i<j<k} [t_{ij},t_{ik}]w_{ij}w_{ik}+[t_{ik},t_{ij}]w_{ik}w_{ij}+[t_{jk},t_{ij}]w_{jk}w_{ij}+\\
+[t_{ij},t_{jk}]w_{ij}w_{jk}+[t_{ik},t_{jk}]w_{ik}w_{jk}+[t_{jk},t_{ik}]w_{jk}w_{ik}.
\end{eqnarray*}

Using now the second relation in the Drinfeld-Kohno Lie algebra, we can rewrite the expression as:

\[2 \sum_{i<j<k}[t_{ij},t_{jk}](w_{ij}w_{jk}+ w_{ki}w_{ij}+ w_{jk}w_{ki}),\]
which vanishes in view of the Arnold relation.

\end{proof}

The KZ connection is a flat connection on the configuration space with values in the Drinfeld-Kohno Lie algebra.
We will see how this connection can be used to construct flat connections on vector bundles on configurations spaces.
Recall that given a finite dimensional complex Lie algebra $\frak{g}$ the Killing form is the symmetric bilinear form $\kappa$ on $\frak{g}$ defined by:
\[\kappa(x,y):= \mathsf{tr}(\mathsf{ad}(x)\circ \mathsf{ad}(y)),\]

where $\mathsf{ad}$ denotes the adjoint representation of $\frak{g}$. The Lie algebra $\frak{g}$ is called semisimple if
the Killing form $\kappa$ is nondegenerate.
We denote by $\U$ the universal enveloping algebra of $\g$. If $\g$ is semisimple then the killing form defines an 
isomorphism
\[\kappa^{\sharp}:\g \rightarrow \g^* ,\]
which in turn induces identifications:

\[\g \otimes \g \cong \g \otimes \g^* \cong \End(\g).\]

We will denote by $\Omega$ the element of $\g \otimes \g$ which corresponds to $\id \in \End(\g)$ under the identification above. Explicitly, one can choose an orthonormal basis $\{I_\mu\}$ for $\g$ with respect to the Killing form and then:
\[\Omega= \sum_\mu I_\mu \otimes I_\mu.\]

The Casimir element of $\g$, denoted by $C$, is the image of $\Omega\in \g \otimes \g $ in the universal enveloping algebra.  Since the Killing form is $\mathsf{ad}$ invariant i.e. $\kappa(\mathsf{ad}(x)(y),z)+ \kappa(y, \mathsf{ad}(x)(z))=0$, the map $\kappa^\sharp: \g \rightarrow \g^*$ is a morphism of representations of $\g$. Since $\id \in \End(\g)$ is an invariant element for the action of $\g$, so is $\Omega$. We conclude that $C$ is a central element of $\U$.

Recall that $\U$ admits a coproduct:
\[ \Delta: \U \rightarrow \U \otimes \U,\]
which is the unique algebra homomorphism with the property that:
\[\Delta(x)= 1 \otimes x + x \otimes 1,\] 
for all $x\in \g$. 
\begin{lemma}\label{lemmaomega}
We regard $\g$ as a subspace of $\U$ via the obvious inclusion. Then:
\[\Omega=\frac{1}{2}(\Delta(C)-1 \otimes C - C\otimes 1).\]
\end{lemma}
\begin{proof}
This is a direct computation:
\begin{eqnarray*}
\Delta(C)=\Delta(\sum_\mu I_\mu  I_\mu)=\sum_\mu \Delta(I_\mu )\Delta(I_\mu)&=& \sum_\mu (1 \otimes I_\mu+ I_\mu \otimes 1)(1 \otimes I_\mu+ I_\mu \otimes 1)\\
&=&1 \otimes C +C\otimes 1 +2 \Omega.
\end{eqnarray*}
\end{proof}

Let $ \iota^{12}: \U \otimes \U \rightarrow \U \otimes \U \otimes \U$ be the map:
\[x\otimes y \mapsto x \otimes y \otimes 1,\]
and define $\iota^{23}, \iota^{13}$ analogously. Then, for $1\leq i <j \leq3$ we set:
\[ \Omega^{ij}:= \iota^{ij} (\Omega).\]

\begin{lemma}\label{lemmadrinfeld}
The following relation is satisfied:
\[[\Omega^{12},\Omega^{23}+ \Omega^{13}]=0.\]
\end{lemma}
\begin{proof}
First, we observe that since $C$ is a central element in $\U$, $1 \otimes 1 \otimes C, 1 \otimes C \otimes 1, C \otimes 1 \otimes 1$ are central elements in $\U \otimes \U \otimes \U$. In view of Lemma \ref{lemmaomega}, we know that for each pair $1 \leq i <j \leq 3$:
\[ \Omega^{ij}=\frac{1}{2}\iota^{ij}(\Delta(C))+X^{ij},\]
where $X^{ij}$ is central. Therefore it suffices to prove that:
\[[\iota^{12}(\Delta(C)),\iota^{23}(\Delta(C))+ \iota^{13}(\Delta(C))]=0.\]
In order to prove this we compute:
\[\iota^{23}(\Delta(C))=\iota^{23}(1 \otimes C + C \otimes 1 +2 \sum_\mu I_\mu \otimes I_\mu)=1\otimes 1 \otimes C +1\otimes C \otimes 1 +2 \sum_\mu1\otimes I_\mu \otimes I_\mu).\]
Similarly:
\[\iota^{13}(\Delta(C))=\iota^{13}(1 \otimes C + C \otimes 1 +2 \sum_\mu I_\mu \otimes I_\mu)=1\otimes 1 \otimes C + C\otimes 1 \otimes 1 +2 \sum_\mu I_\mu \otimes 1\otimes I_\mu).\]
Therefore:
\begin{eqnarray*}
\iota^{13}(\Delta(C))+\iota^{23}(\Delta(C))&=& 2  \sum_\mu \Delta(I_\mu)\otimes I_\mu + X,
\end{eqnarray*}
where $X$ is central.
Finally we compute:
\[\frac{1}{2}[\iota^{12}(\Delta(C)),\iota^{23}(\Delta(C))+ \iota^{13}(\Delta(C))]=[\Delta(C)\otimes 1,\sum_\mu \Delta(I_\mu)\otimes I_\mu]=\sum_\mu [\Delta(C), \Delta (I_\mu )]\otimes I_\mu = 0.\]
\end{proof}

The previous lemma will be the key to constructing an action of the pure braid group $P_n$ on tensor products of representations of a complex semisimple Lie algebra. 
Given a representation 
$\rho: \g \rightarrow \End(V)$ of a Lie algebra $\g$, 
we will also denote by $\rho$ the corresponding homomorphism of associative algebras $\rho: \U \rightarrow \End(V)$. 

\begin{lemma}\label{lemmakey}
Let $\frak{g}$ be a finite dimensional complex semisimple Lie algebra and $\rho_1: \g \rightarrow \End(V_1), \dots,\rho_n: \g \rightarrow  \End(V_n)$ be representations of $\frak{g}$. Then there is a homomorphism of Lie algebras \[\Upsilon_n:\frak{t}_n \rightarrow \End(V_1 \otimes \dots \otimes V_n)\] given by the formula:
\[t_{ij}\mapsto (\rho_1\otimes \dots \otimes \rho_n)\circ \lambda^{ij}(\Omega)\in \End(V_1)\otimes \dots \otimes \End(V_n)\subset \End(V_1 \otimes \cdots \otimes V_n), \]
where $ \lambda^{ij}: \U \otimes \U \rightarrow \U^{\otimes n}$ is the morphism of algebras given by:
\[x \otimes y \mapsto 1 \otimes \dots \otimes 1 \otimes \underbrace{x}_i \otimes 1 \otimes \dots \otimes 1 \otimes \underbrace{ y}_j \otimes 1 \otimes \dots \otimes 1.\]
\end{lemma}
\begin{proof}
We need to prove that the endomorphisms $\Upsilon(t_{ij})$ satisfy the Khono-Drinfeld relations. It is clear from the definition that 
\[[\Upsilon(t_{ij}), \Upsilon(t_{kl})]=0,\]
if $ \#\{i,j,k,l\}=4$. It remains to prove that: 
\[ [\Upsilon(t_{ij}), \Upsilon(t_{ik})+ \Upsilon(t_{ik})]=0,\]
if $\#\{ i,j,k\}=3$. Clearly, it is enough to consider the case $n=3$. Since $\rho_1\otimes \dots \otimes \rho_n$ is a morphism of algebras, it suffices to prove that:
\[[\Omega^{12},\Omega^{23}+ \Omega^{13}]=0,\]
which is precisely the claim of  Lemma \ref{lemmadrinfeld}.
\end{proof}

Lemma \ref{lemmakey} together with the flatness of the KZ connection imply the following result:

\begin{theorem}\label{theorempurebraids}
Let $\frak{g}$ be a finite dimensional complex semisimple Lie algebra and $\rho_1: \g \rightarrow \End(V_1), \dots,\rho_n: \g \rightarrow  \End(V_n)$ be representations of $\frak{g}$. The vector space $(V_1 \otimes \cdots \otimes V_n)$ has the structure of a representation of the  pure braid group $P_n$, given by holonomy of the connection 
$\Upsilon_n (\theta_n)$ on the trivial vector bundle over $\hat{C}_n(\mathbb{C})$ with fiber $(V_1 \otimes \cdots \otimes V_n)$.
\end{theorem}
\begin{proof}
Since the KZ connection $\theta_n$ is flat and $\Upsilon_n$ is a morphism of Lie algebras, the connection $\Upsilon_n (\theta_n)$ is a flat connection on the trivial vector bundle over $\hat{C}_n(\mathbb{C})$. Using the identification
\[P_n \cong \pi_1 (\hat{C}_n(\mathbb{C}), {\bf p}),\]
one obtains the desired representation. 
\end{proof}

The construction above can be symmetrized to obtain representation of the braid groups $B_n$.

\begin{theorem}\label{theorembraids}
Let $\frak{g}$ be a finite dimensional complex semisimple Lie algebra and $\rho: \g \rightarrow \End(V)$ be representation of $\frak{g}$. The vector space $V^{\otimes n}$ has the structure of a representation of the braid group $B_n$.
\end{theorem}

\begin{proof}
By applying the construction in Theorem \ref{theorempurebraids} to the case $\rho_i=\rho$, one obtains a flat connection $\Upsilon_n (\theta_n)$ on the trivial vector bundle over $\hat{C}_n(\mathbb{C})$ with fiber $V^{\otimes n}$. The symmetric group $\Sigma_n$ acts on $V^{\otimes n}$ and also on $\hat{C}_n(\mathbb{C})$ by permuting the coordinates.
 Therefore it also acts diagonally on $\Omega(\hat{C}(\mathbb{C}))\otimes \End(V^{\otimes n})$. One easily checks that $\Upsilon_n (\theta_n)$ is invariant under this diagonal action, and
therefore descends to a flat connection form on the vector bundle $ (\hat{C}_n(\mathbb{C}) \times V^{\otimes n})/\Sigma_n$ over 
$\hat{C}_n(\mathbb{C})/\Sigma_n \cong C_n(\mathbb{C})$. The holonomy of this connection gives a representation of
$\pi_1(C_n(\mathbb{C}), {\bf p})\cong B_n$ on the vector space $V^{\otimes n}$.
\end{proof}

\subsection{Representations of $\frak{sl_2}(\mathbb{C})$ and Lawrence-Krammer-Bigelow }

We will consider the KZ connection for the semisimple Lie algebras $\g=\frak{sl}_2(\mathbb{C})$ of $2 \times 2$ complex matrices of trace zero. Let us fix the following basis for $\slt$:

\begin{equation*}
H= \left( \begin{array}{cc} 1 &0\\
0&-1\\
\end{array}\right)
, \,\,E = \left( \begin{array}{cc} 0 &1\\
0&0\\
\end{array}\right)
,\,\, F= \left( \begin{array}{cc} 0 &0\\
1&0\\
\end{array}\right).
\end{equation*}
In terms of the basis, the bracket is given by:
\[
[H,E]=2E, \,\,[H,F]=-2F,\,\, [E,F]=H.\]

\begin{definition}
Let $\lambda \in \mathbb{C}$ be a complex number. The Verma module $M_\lambda$ with highest weight $\lambda$ is the representation of $\slt$ defined as follows. As a vector space it  is generated by the elements $F^j(v_\lambda)$ for $j \geq 0$.
The action of $\slt$ on $M_\lambda$ is given by:
\begin{eqnarray*}
 F (F^j(v_\lambda))&=&F^{j+1}(v_\lambda),\\   E (F^j(v_\lambda))&=&j(\lambda-j+1)F^{j-1}(v_\lambda),\\ 
 H (F^j(v_\lambda))&=&
(\lambda -2j)F^{j}(v_\lambda).
\end{eqnarray*}

Given $\Lambda=(\lambda_1, \dots ,\lambda_n) \in \mathbb{C}^n$ we set $|\Lambda|=\lambda_1 + \dots + \lambda_n$,
and we define the space of weight vectors with weight $|\Lambda| -2m$ to be the vector space:

\[W[|\Lambda|-2m]:=\{x \in M_{\lambda_1} \otimes \dots \otimes M_{\lambda_n}: H(x)=(|\Lambda|-2m)x\}.\]
The space of nullvectors $N[|\Lambda|-2m]$ is:
\[N[|\Lambda|-2m]:=\{x \in W[|\Lambda|-2m]: E(x)=0\}.\]
\end{definition}

\begin{lemma}
The vector space $W[|\Lambda|-2m]$ is finite dimensional, of dimension equal to the number of ordered partitions of length $n$ of the number $m$.
\end{lemma}
\begin{proof}
Since $H$ acts on $M_\lambda$ in a diagonal manner with respect to the given basis, a sum of basis elements is in 
$W[|\Lambda|-2m]$ if and only if each of the basis elements is. Therefore, it suffices to count the number of basis elements in $W[|\Lambda|-2m]$. A basis element $x= F^{j_i}\otimes\dots \otimes F^{j_n}$ belongs to $W[|\Lambda|-2m]$ precisely when:
\[ (|\Lambda|-2m)x=H(x)=\sum_i (\lambda-2j_i)x,\]
that is, precisely when $\sum_i j_i=m$. 
\end{proof}

Given a nonzero complex number $\tau$ we define the KZ connection with parameter $\tau$, \[\theta_{n,\tau}:=\frac{1}{\tau}\theta_n.\] Note that since $d\theta_n=0=[\theta_n,\theta_n]$, the connection $\theta_{n,\tau}$ is flat. We now fix another complex number $\lambda \in \mathbb{C}$ and consider the representation $M^{\otimes n}_\lambda$. Using the notation in Theorem \ref{theorempurebraids}, the differential form $\Upsilon_n (\theta_{n,\tau})$ is a flat connection 
on $M^{\otimes n}_\lambda$.  

\begin{proposition}
The flat connection $\Upsilon_n \theta_{n,\tau}$ commutes with the action $\rho: \slt \rightarrow \End(M^{\otimes n}_\lambda)$, i.e. for any vector field $Y$ on the configuration space $\hat{C}_n(\mathbb{C})$:

\[ \rho(v) (\Upsilon_n \theta_{n,\tau}(Y)(x))=\Upsilon_n \theta_{n,\tau}(Y)(\rho(v)(x)),\]
for any $v \in \slt$ and $x \in M^{\otimes n}_\lambda$.
The vector space $N[n\lambda-2m]$ is a representation \[\mu_{n,m,\lambda, \tau}: B_n\rightarrow \Aut(N[n\lambda-2m])\] of the braid group $B_n$.
\end{proposition}
\begin{proof}
It suffices to prove that each of the endomorphisms $\Upsilon_n(t_{ij})$ commutes with $\rho$.
The action of $\slt$ on $M^{\otimes n}_\lambda$ is given by $\rho^{\otimes n}\circ  \Delta^{ n}(x): \g \subset \U \rightarrow \End(M^{\otimes n}_\lambda).$ Therefore, it suffices to prove that each element $\Omega^{ij}=\lambda^{ij}(\Omega)$ commutes with $\Delta^{ n}(v)$. By Lemma \ref{lemmaomega} we know that $\Omega$ differs from $\Delta(C)$ by a central element, so it suffices to prove that \[[\lambda^{ij}(\Delta(C)), \Delta^{ n}(v)]=0,\] which is a consequence of the fact that $C$ is a central element of $\U$. This completes the proof of the first claim.
Since the connection $\Upsilon_n \theta_{n,\tau}$ commutes with the action of $\rho$, then it preserves the finite dimensional vector space $N[n\lambda-2m],$ and therefore it restricts to a flat connection on the trivial vector bundle over $\hat{C}_n(\mathbb{C})$ with fiber $N[n\lambda-2m]$ . As before, the connection is invariant under the action of $\Sigma_n$ and therefore it descends to a flat connection on the quotient vector bundle over $C_n(\mathbb{C})$. The holonomy of this connection gives $N[n\lambda-2m]$ the structure of a representation of $B_n$.
\end{proof}

We are now ready to state the following remarkable theorem of Khono  which describes the relation between the Lawrence-Krammer-Bigelow representations and the representations on the nullpaces of Verma modules of $\slt$ given by holonomy of the KZ connection.

\begin{theorem}[Kohno \cite{Kohno2}]
There exists an open dense subset $U \subset \mathbb{C}^2$, such that for $(\tau,\lambda)\in U$ the
representation \[ \mu_{n,2,\lambda,\tau}: B_n \rightarrow \Aut(N[n\lambda-4])\] is equvalent to the representation
obtained from the Lawrence-Krammer-Bigelow representation \[B_n \rightarrow \Aut_R(\mathcal{H})\cong \Aut(R^{\frac{n(n-1)}{2}}),\] by setting:
\[ q=\exp^{-2\pi i \frac{\lambda}{\tau}}, \,\, t= \exp^{\frac{2\pi i}{\tau}}.\]
Here, as before, $R$ is the ring $R=\mathbb{Z}[q,q^{-1},t,t^{-1}]$.
\end{theorem}

\section{The Yang-Baxter equation and the Drinfeld-Kohno theorem}\label{section5}
\subsection{ The Yang-Baxter equation and quasi-triangular bialgebras}
The Yang-Baxter equation was originally introduced in the field of statistical mechanics. Solutions to this equation provide a systematic way to construct representations of the braid groups $B_n$. We will introduce the notion of quasi-triangular bialgebra, and explain that these are algebraic structures with the property that
their modules are naturally endowed with solutions to the Yang-Baxter equation.

\begin{definition}
Let $V$ be a complex vector space. An automorphism $R$ of $V \otimes V$ is called an $r$-matrix if it satisfies the  Yang-Baxter equation:
\[(R \otimes \id) \circ (\id \otimes R)\circ (R \otimes \id)=(\id \otimes R)\circ (R \otimes \id)\circ (\id \otimes R)\in \End(V^{\otimes 3}). \]
\end{definition}

\begin{lemma}
Given an $r$-matrix $R \in \Aut(V\otimes V)$, there is a representation \[\rho_R: B_n \rightarrow \Aut(V^{\otimes n}),\]
defined by:

\[\sigma_i \mapsto \id^{i-1}\otimes R \otimes \id^{n-i-1}.\]
\end{lemma}
\begin{proof}
The identity \[ \rho(\sigma_i)\rho(\sigma_{i+1})\rho(\sigma_i)=\rho(\sigma_{i+1})\rho(\sigma_i)\rho(\sigma_{i+1}),\]
follows from the fact that $R$ satisfies the Yang-Baxter equation. The relation $\sigma_i \sigma_j= \sigma_j \sigma_i$ if $|i-j|>2$ is clear from the form of the map $\rho_R$.
\end{proof}

Let us now introduce some algebraic definitions.

\begin{definition}
A bialgebra is a tuple $(A, \mu, \eta, \Delta, \epsilon)$ where:
\begin{itemize}
\item $A$ is a vector space over $\mathbb{C}$.
\item $\mu$ is  map $A \otimes A \rightarrow A$, called the product.
\item $\eta$ is a map $\mathbb{C}\rightarrow A$, called the unit.
\item $\Delta$ is a map $ \Delta: A \rightarrow A \otimes A$, called the coproduct.
\item $\epsilon$ is a map $ A \rightarrow \mathbb{C}$, called the counit.
\end{itemize}
Satisfying the conditions:
\begin{itemize}
\item $(A,\mu, \eta)$ is a unital associative algebra.
\item $(A,\Delta, \epsilon)$ is a counital coassociative coalgebra.
\item $ \Delta: A \rightarrow A \otimes A$ is a morphism of associative algebras.
\end{itemize}
Of course, we will abuse the notation and simply say that $A$ is a bialgebra.
\end{definition}
\begin{example}
We have already encountered an important example of a bialgebra, the universal enveloping
algebra $\U$ of a Lie algebra $\g$. Recall that the coproduct is the unique algebra homomorphism
$\Delta: \U \rightarrow \U \otimes \U$, with the property that $ \Delta(x)= 1 \otimes x + x \otimes 1$, for any
$x \in \g$. The bialgebra $\U$ is cocommutative but in general it is not commutative.
\end{example}


 Given  modules $V,W$ of a bialgebra $A$,
the vector space $V \otimes W$, is also a module of $A$ with action given by:
\[a (v\otimes w):= \Delta(a)(v \otimes w).\]
Thus, bialgebras are algebras for which the category of modules admits a tensor product. In general, the modules $V\otimes W$ and $ W \otimes V$ need not be isomorphic. In case the coproduct $\Delta$ is cocommutative,
the linear map \[ \tau_{V,W}: V \otimes W \rightarrow W \otimes V,\]
\[ v\otimes w \mapsto w \otimes v,\]
is an isomorphism of $A$-modules. The condition that $A$ is cocommutative is sufficient, but not necessary in order to have consistent isomorphism of $A$-modules $V \otimes W \cong W \otimes V$. This consideration leads to the following definition:

\begin{definition}
An almost-cocommutative bialgebra is a bialgebra $(A, \mu, \eta, \Delta, \epsilon)$ together with an invertible element
$R \in A\otimes A$, called the universal $r$-matrix,  such that for any $a \in A$:
\[ \tau_{A,A} \circ \Delta(a)= R \Delta(a) R^{-1},\]
where $\tau_{A.A}: A \otimes A \rightarrow A \otimes A$ is given by $a \otimes b \mapsto b \otimes a$.
\end{definition}

Given an almost-cocommutative bialgebra, and modules $V, W$ the map \[ \tau_{R,V,W}:=\tau_{V,W} \circ R: V \otimes W \rightarrow W\otimes V,\] is an isomorphism of $A$-modules. Thus, an almost-cocommutative bialgebra is a bialgebra for which there are natural isomorphisms of $A$-modules $ V\otimes W \cong W \otimes V$.

 Let us now consider $A$-modules $V,W,Z$. There are several ways to obtain isomorphisms of $A$-modules between $ V\otimes W \otimes Z$ and $W \otimes V \otimes Z$, and it is natural to expect that these isomorphisms should coincide. Namely, one would expect the following diagrams to commute:
\begin{equation}\label{diagram1}
\xymatrix{
V \otimes W \otimes Z \ar[rr]^{\tau_{R,V,W\otimes Z}} \ar[d]_{\tau_{R,V,W}\otimes \id} & & W \otimes Z \otimes V \ar[d]^{\id}\\
W \otimes V \otimes Z   \ar[rr]^{\id \otimes \tau_{R,V,Z}} && W\otimes Z \otimes V
}
\end{equation}

\begin{equation}\label{diagram2}
\xymatrix{
V \otimes W \otimes Z \ar[rr]^{\tau_{R,V\otimes W,Z}} \ar[d]_{\id \otimes \tau_{R,W,Z}} & & Z \otimes V \otimes W \ar[d]^{\id}\\
V \otimes Z \otimes W   \ar[rr]^{ \tau_{R,V,Z}\otimes \id} && Z\otimes V \otimes W
}
\end{equation}

The commutativity of these diagrams can be assured by imposing conditions on the universal $r$-matrix $R$.
This leads to the following definition. 

\begin{definition}
A quasi-triangular bialgebra is an almost-cocommutative bialgebra $A$ such that the universal $r$-matrix $R\in A \otimes A$ satisfies the equations:
 \[ (\Delta \otimes \id) R=R_{13} R_{23}.\]
\[ (\id \otimes \Delta) R = R_{13} R_{12}.\]
Here $R_{12}=R \otimes 1, R_{23}=1 \otimes R$ and $R_{13}= (\tau_{A,A}\otimes \id) (R_{23})=(\id \otimes \tau_{A,A})(R_{12}) $.
\end{definition}

\begin{example}
The simplest example of a quasi-triangular bialgebra is a cocommutative bialgebra, with $r$-matrix $R=1 \otimes 1$. Thus, the universal enveloping algebra $\U$ of a Lie algebra $\g$ is a quasi-triangular bialgebra.
\end{example}

\begin{remark}
The definition of a quasi-triangular bialgebra is designed to guarantee that the category of modules has a well
behaved tensor product. Namely, a tensor product with natural isomorphisms \[C_{VW}:V \otimes W \cong W \otimes V,\]
for which the diagrams (\ref{diagram1}) and (\ref{diagram2}) commute. This kind of category is known as a braided monoidal category.   Observe that in diagrams (\ref{diagram1}) and (\ref{diagram2}) expressions of the type 
$V \otimes W \otimes Z$ appear. This notation is using implicitly that the natural isomorphisms of vector spaces:
\[(V \otimes W) \otimes Z \cong  V \otimes (W\otimes Z)\]
are isomorphism of representations. 
In more generality, one may consider the case where there are natural isomorphisms of representations
\[(V\otimes W)\otimes Z \cong V \otimes (W \otimes Z), \]
which are not the obvious isomorphisms of vector spaces.
This corresponds to the general definition of braided monoidal category. The algebraic structure on $A$ that guarantees that the category of modules is a braided monoidal category is that of a quasi-triangular quasi-bialgebra. Note that the fact that for a bialgebra the obvious linear isomorphism
$(V \otimes W) \otimes Z \cong V \otimes (W\otimes Z)$ is a morphism of modules corresponds to the fact that the coproduct is coassociative. Thus, in a quasi-triangular quasi-bialgebra one does not require $\Delta$ to be strictly coassociative. Instead, there is an element $ \phi \in A\otimes A \otimes A$, called the Drinfeld associator, which controls the lack of associativity much in the same way in which $R$ controls the lack of commutativity. We will not discuss the details of these elegant structures here. The interested reader may consult Kasssel's book \cite{Kassel}.
\end{remark}

\begin{proposition}
Let $A$ be a quasi-triangular bialgebra with universal $r$-matrix $R\in A \otimes A$. Given $A$-modules $V,W,Z$ there are isomorphisms of $A$-modules:
\[\tau_{R,V,W}: V\otimes W \rightarrow W\otimes V,\]
defined  by:

\[ \tau_{R,V,W}:= \tau_{V,W} \circ R,\]
satisfying the equations:

\begin{equation}\label{eq1}
(\tau_{R,V\otimes W, Z})= ( \tau_{R, V,Z}\otimes \id) \circ ( \id \otimes \tau_{R,W,Z }),
\end{equation}
\begin{equation}\label{eq2}
(\tau_{R,V, W\otimes Z})= (\id \otimes \tau_{R,V,Z }) \circ (\tau_{R, V,W}\otimes \id),
\end{equation}
\begin{equation}\label{eq3}
(\tau_{R,W,Z}\otimes \id)\circ (\id \otimes \tau_{R,V,Z} ) \circ (\tau_{R,V,W}\otimes \id)=(\id \otimes \tau_{R,V,W}). \circ(\tau_{R,V,Z}\otimes \id)\circ (\id \otimes \tau_{R,W,Z}).
\end{equation}
\end{proposition}
\begin{proof}
Let us first check that $\tau_{R,V,W}$ is an isomorphism of $A$-modules. Since $R$ is invertible, $\tau_{R,V,W}$ is an invertible linear map. Let us check that it is $A$-linear:

\begin{eqnarray*}
\tau_{R,V,W}(a (v\otimes w))&=&\tau_{R,V,W}(\Delta(a)(v\otimes w))= \tau_{V,W} \circ R \circ \Delta(a)(v\otimes w)\\
&=& \tau_{V,W} \circ \tau_{A,A}(\Delta(a)) \circ R (v\otimes w)=\Delta(a) \circ \tau_{V,W} \circ R (v\otimes w)\\
&=&a \circ \tau_{R,V,W}(v\otimes w).
\end{eqnarray*}
Equations (\ref{eq1}) and (\ref{eq2}) are analogous. Let us prove equation (\ref{eq1}). For this we use the fact
that $(\Delta \otimes \id)(R)=R_{13} R_{23}$, and compute:

\begin{eqnarray*}
 ( \tau_{R, V,Z}\otimes \id) \circ ( \id \otimes \tau_{R,W,Z })&= &(\tau_{V,Z} \circ R \otimes \id) \circ (\id \otimes \tau_{W,Z}\circ R )\\
 &=&(\tau_{V,Z}\otimes \id) \circ R_{12} \circ (\id \otimes \tau_{W,Z})\circ R_{23}\\
 &=&(\tau_{V,Z}\otimes \id) \circ (\id \otimes \tau_{W,Z}) (\id \otimes \tau_{A,A})(R_{12})\circ R_{23}\\
  &=&(\tau_{V,Z}\otimes \id) \circ (\id \otimes \tau_{W,Z}) R_{13}\circ R_{23}\\
  &=&(\tau_{V,Z}\otimes \id) \circ (\id \otimes \tau_{W,Z}) (\Delta \otimes \id)R\\
  &=&(\tau_{V\otimes W,Z})\circ (\Delta \otimes \id) R\\
  &=& \tau_{R,V\otimes W,Z}. 
\end{eqnarray*}

Let us now prove equation (\ref{eq3}). We compute:
\begin{eqnarray*}
(\tau_{R,W,Z}\otimes \id)&\circ& (\id \otimes \tau_{R,V,Z} ) \circ (\tau_{R,V,W}\otimes \id)=\\
&=&(\tau_{W,Z}\otimes \id)\circ R_{12} \circ (\id \otimes \tau_{V,Z} )\circ R_{23} \circ (\tau_{V,W}\otimes \id)\circ R_{12}\\
&=&(\tau_{W,Z}\otimes \id)\circ R_{12} \circ (\id \otimes \tau_{V,Z} ) \circ (\tau_{V,W}\otimes \id)\circ R_{13}\circ R_{12}\\
&=&(\tau_{W,Z}\otimes \id) \circ (\id \otimes \tau_{V,Z} )\circ R_{13} \circ (\tau_{V,W}\otimes \id)\circ R_{13}\circ R_{12}\\
&=&(\tau_{W,Z}\otimes \id) \circ (\id \otimes \tau_{V,Z}) \circ (\tau_{V,W}\otimes \id)\circ R_{23} \circ R_{13}\circ R_{12}\\
&=&(\tau_{W,Z}\otimes \id) \circ (\id \otimes \tau_{V,Z}) \circ (\tau_{V,W}\otimes \id)\circ R_{23} \circ R_{13}\circ R_{12}.\\
\end{eqnarray*}

On the other hand:
\begin{eqnarray*}
(\id \otimes \tau_{R,V,W}) &\circ& (\tau_{R,V,Z}\otimes \id)\circ (\id \otimes \tau_{R,W,Z})=\\
&=& (\id \otimes \tau_{V,W})\circ R_{23} \circ (\tau_{V,Z}\otimes \id)\circ R_{12}\circ (\id \otimes \tau_{W,Z})\circ R_{23}\\
&=&(\id \otimes \tau_{V,W})\circ R_{23} \circ (\tau_{V,Z}\otimes \id)\circ (\id \otimes \tau_{W,Z})\circ R_{13}\circ R_{23}\\
&=&(\id \otimes \tau_{V,W}) \circ (\tau_{V,Z}\otimes \id) \circ R_{13}\circ (\id \otimes \tau_{W,Z})\circ R_{13}\circ R_{23}\\
&=&(\id \otimes \tau_{V,W}) \circ (\tau_{V,Z}\otimes \id) \circ (\id \otimes \tau_{W,Z})\circ R_{12}\circ R_{13}\circ R_{23}.\\
\end{eqnarray*}

Clearly,
\[(\tau_{W,Z}\otimes \id) \circ (\id \otimes \tau_{V,Z}) \circ (\tau_{V,W}\otimes \id)=(\id \otimes \tau_{V,W}) \circ (\tau_{V,Z}\otimes \id) \circ (\id \otimes \tau_{W,Z}).\]
Therefore, it suffices to prove that:
\[R_{23}  R_{13}R_{12}=R_{12} R_{13} R_{23}.\]
For this we compute:

\begin{eqnarray*}
R_{12} R_{13} R_{23}&=&R_{12}(\Delta \otimes \id )R\\
&=& (R \otimes 1)( \Delta \otimes \id)R\\
&=&(\tau_{A,A} \otimes \id) (\Delta  \otimes \id) (R) R_{12}\\
&=&(\tau_{A,A}\otimes \id) (R_{13} R_{23}) R_{12}\\
&=&R_{23}R_{13} R_{12}.
\end{eqnarray*}
This completes the proof.
\end{proof}

\begin{corollary}
Let $A$ be a quasi-triangular bialgebra with universal $r$-matrix $R$, and $V$ a module over $A$. Then $\tau_{R,V,V}\in \Aut_A(V)$ is a solution of the Yang-Baxter equation. In particular, the vector space $V^{\otimes n}$ is a representation of the braid group $B_n$.
\end{corollary}

\subsection{The Drinfeld-Kohno theorem}
We have seen that modules over quasi-triangular bialgebras produce representations of the braid groups. However, so far we have not encountered nontrivial examples of quasi-triangular bialgebras. A rich source of examples comes from Lie theory: 
the universal enveloping algebra $\U$ of a complex semisimple Lie algebra $\g$ can be deformed to obtain a quantum enveloping algebra. These provide interesting examples of quasi-triangular bialgebras and of representations of the braid groups. 

In order to describe the Drinfeld-Jimbo bialgebras, it will be necessary to review some facts regarding the classification of semisimple Lie algebras. A good reference for this subject is Humphreys' book \cite{Hum}. Let $\g$ be a finite dimensional complex semisimple Lie algebra. A Cartan subalgebra $\h \subset \g$ is an abelian subalgebra with the property that $N_\g(\h)=\h$. Where $N_\g(\h)$ is the normalizer of $\h$ in $\g$, i.e:
\[N_\g(\h):= \{ x \in \g: [\h,x]\subset \h \}.\]

Cartan subalgebras always exist and moreover, any two Cartan subalgebras are conjugated to one another. Let us now on fix a Cartan subalgebra $\h \subset \g$. For any element $x \in \h$, the linear map $\ad(x): \g \rightarrow \g$ is diagonalizable. Since $\h$ is abelian, all the maps $\ad(x)$ commute and therefore $\g$ decomposes as a direct sum of eigenspaces:
\[ \g = \bigoplus_{\alpha \in \h^*} \g_\alpha, \]
for some functionals $\alpha \in \h^*$ so that:

\[ [h,y]= \alpha(h)y, \text{ if } y \in \g_\alpha , h \in \h.\]

The set of roots, denoted by $\Delta$, is the set of nonzero $\alpha \in \h^*$ in the decomposition above. 

\begin{lemma}
Let $\g$ be a finite dimensional complex semisimple Lie algebra and $\h \subset \g$ a Cartan subalgebra. For any $\alpha \in \h^*$ set
\[ \g_\alpha:= \{ x \in \g: [h,x]= \alpha (h)x \text{ for all } h \in \h\}.\]
Then:

\[[\g_\alpha,\g_\beta] \subset \g_{\alpha+\beta}.\]
If $x \in \g_\alpha$ for $\alpha \neq 0$, $\ad(x)$ is nilpotent. If $\alpha + \beta \neq 0 $ then $\g_\alpha$ is orthogonal to $\g_\beta$ with respect to the Killing form. In particular, the restriction of $\kappa$ to $\h = \g_0$ is nondegenerate. 
\end{lemma}
\begin{proof}
The first claim follows from direct computation. Take $h \in \h, x \in \g_\alpha$ and $y \in \g_\beta$:

\[ [h [x,y]]=[[h,x],y] + [x,[h,y]]=(\alpha(x)+ \beta(x)) [x,y].\] 
For the second claim, consider $y \in \g_\beta$, then:
\[ \mathsf{ad}(x)^n(y)\in \g_{n\alpha + \beta} ,\]
since $\g_\gamma$ is nonzero only for finitely many $\gamma$, we conclude that $ \ad(x)^n=0$ for $n$ sufficiently large. Let us now prove the last statement. We choose a basis for $\g$ compatible with the decomposition
$g= \bigoplus_\gamma \g_\gamma$ then if $x \in \g_\alpha$ and $ y \in g_\beta$ then 
\[\ad(x)\circ \ad(y): \g_\gamma \rightarrow \g_{\gamma+ \alpha +\beta},\]
so if $ \alpha +\beta \neq 0$ the matrix associated to $ \ad(x) \circ \ad(y)$ has zeros in the diagonal and therefore
\[\kappa(x,y)=0.\]

\end{proof}

The set $\Delta\subset \h^*$ of roots of $\g$ has a beautiful combinatorial structure which makes it an abstract
root system. Let us recall this definition:

\begin{definition}
Let $E$ be a finite dimensional real vector space with a symmetric positive definite bilinear form $(\,,\,)$.
A root system in $E$ is a finite subset $\Delta \subset E$ which spans $E$ and does not containing zero, such that:
\begin{enumerate}
\item If $\alpha \in \Delta$ then $-\alpha \in \Delta$, and no other multiple of $\alpha$ is in $\Delta$.
\item For each $ \alpha \in \Delta$ the reflection $\sigma_\alpha$ with respect to the plane orthogonal to $\alpha$ fixes the set $\Delta$.
\item If $\alpha, \beta \in \Delta$ then \[\frac{2(\alpha, \beta)}{(\alpha,\alpha)}\in \mathbb{Z}.\]
\end{enumerate}
\end{definition}

Coming back to the roots of the Lie algebra $\g$, since the Killing form $\kappa$ is nondegenerate when restricted to $\h$, it induces a bilinear form in $\h^*$ which we still denote by $\kappa$.

\begin{proposition}
The set $\Delta$ spans a real vector space $E \subset \h^*$ of dimension equal to the complex dimension of $\h^*$. Moreover, the bilinear form $\kappa$ is positive definite when restricted to $E$, and the set $\Delta$ is an abstract root system in $E$.
\end{proposition}
\begin{proof}
The interested reader may find the proof of this proposition in Humphreys' book \cite{Hum}.
\end{proof}

The complete structure of an abstract root system $\Delta \subset E$ can be described by its Cartan matrix, which is
defined as follows. Given $ v \in E$ denote by $T_v$ the space of vectors orthogonal to $v$. We say that $v \in E $ is singular if $v \in T_\alpha$ for some $\alpha \in \Delta$. We say that $v \in E$ is regular if it is not singular.
Clearly, regular vectors exist since $E$ is not the union of finitely many hyperplanes. Given a regular element $v \in E$,
the set $\Delta$ decomposes into positive and negative roots :
\[ \Delta = \Delta^{+}  \coprod \Delta^{-},\]
where 

\[ \Delta^+:= \{ \alpha \in \Delta: (v,\alpha)>0\} \text{ and } \Delta^-:= \{ \alpha \in \Delta: (v,\alpha)<0\}.\]

 We say that a positive root $\alpha$ is simple if it cannot be written in the form  $\alpha = \beta_1 + \beta_2$ where
 $\beta_1, \beta_2$ are positive roots. 
 
 \begin{definition}
 Let $\Delta$ be an abstract root system in $E$ and $v \in E$ a regular element. Fix an ordering $\alpha_1, \dots , \alpha_l$ of the simple roots. The Cartan matrix of $\Delta$ with respect to $v$ is the matrix:
 \[ C_{i,j}:= \frac{2 (\alpha_i, \alpha_j)}{(\alpha_i, \alpha_i)}.\]
 \end{definition}
 
 \begin{proposition}
 The Cartan matrix $C$ of the root system $\Delta$ is well defined, up to conjugation by a permutation matrix.
 Moreover, $C$ has the following properties:
 \begin{enumerate}
 \item The entries $C_{ij}$ are non-positive integers if $i\neq j$  and $C_{ii}=2$.
 \item There exists a unique diagonal matrix  $D=\mathsf{Diag}(d_1, \dots, d_l )$ with $d_i \in \{ 1,2,3\}$, such that $DC$ is symmetric and positive definite.
 \end{enumerate}
 \end{proposition}
 \begin{proof}
 The interested reader can find the proof in Humphreys' book \cite{Hum}.
 \end{proof}

We are ready to define the quantum enveloping algebra of a complex semisimple Lie algebra.
These algebras are topological in the sense that all tensor products appearing in the definition are topological tensor products, which we now define. We denote by $\mathbb{C}[[h]]$ the algebra of formal power series in the variable $h$ with complex coefficients.
A module $M$ over $\mathbb{C}[[h]]$ has the structure of a topological vector space: a basis of neighborhoods for zero 
is $\{h^n M\}_{n \geq 0}$. 
\begin{definition}
Let $M$ and $N$ be $\mathbb{C}[[h]]$-modules. The topological tensor product of $M$ and $N$, denoted $M \hat{\otimes} N$ is the $\mathbb{C}[[h]]$-module:
\[ M \hat{\otimes} N:=\varprojlim \frac{ M \otimes_{\mathbb{C}[[h]]} N}{h^n (M \otimes_{\mathbb{C}[[h]]} N)}.\]
There is a natural map $M \otimes_{\mathbb{C}[[h]]}N \rightarrow M \hat{\otimes} N$.

\end{definition}

\begin{definition}
A topological quasi-triangular bialgebra is a $\mathbb{C}[[h]]$-module $A$ together with structure maps and universal $r$-matrix as in the definition of quasi-triangular bialgebra, where all tensor products are replaced by topological tensor products and structure maps are required to be continuous.
\end{definition}

The following notation will be useful in describing the Drinfeld-Jimbo construction. Given an invertible element $q$ in an algebra $A$ and a positive integer $n$, we set:

\[ [n]_q:= \frac{q^n-q^{-n}}{q-q^{-1}},\]

\[ [n]_q!:= [n]_q [n-1]_q \dots [1]_q,\]

\[ \binom{n}{k}_q := \frac{[n]_q! }{[n-k]_q![k]_q!}.\]

\begin{definition}
Let $\g$ be a finite dimensional complex semisimple Lie algebra with Cartan matrix $C$ of size $l \times l$. The Drinfeld-Jimbo algebra $\mathsf{U}_h(\g)$ is the quotient of the free algebra over $\mathbb{C}[[h]]$ on the set $\{H_i, E_i, F_i\}_{i \leq l}$ by the closure of the two sided ideal generated by the following relations:
\begin{eqnarray*}
 &\,&[H_i, H_j]=0,\\
&\,& [E_i, F_j]-\delta_{ij}\left( \frac{\sinh (\frac{h d_i H_i}{2})}{\sinh({\frac{hd_i}{2}})} \right)=0,\\
&\,&[H_i, E_j]-C_{ij}E_j=0,\\
&\,&[H_i, F_j]+C_{ij} F_j=0,\\
&\,&\sum_{k=0}^{1-C_{ij}}(-1)^k \binom{1-C_{ij}}{k}_{q_i} {E_i}^k {E_j} E_i^{1-C_{ij}-k}=0, \text{ for } i \neq j,\\
&\,&\sum_{k=0}^{1-C_{ij}}(-1)^k \binom{1-C_{ij}}{k}_{q_i} {F_i}^k {F_j} F_i^{1-C_{ij}-k}=0, \text{ for } i \neq j.\\
\end{eqnarray*}
Here $q_i:= \exp{(\frac{hd_i}{2})}$.
\end{definition}

\begin{theorem}[Drinfeld \cite{Drinfeld2}-Jimbo \cite{Jimbo}]
The algebra $\mathsf{U}_h(\g)$ is a topological bialgebra with structure maps characterized by the following properties:
\begin{eqnarray*}
&\,& \Delta_h(H_i)=H_i \otimes 1 + 1 \otimes H_i,\\
&\,& \Delta_h(E_i)=E_i \otimes \exp(\frac{hd_i H_i}{4})+ \exp(\frac{-hd_i H_i}{4})\otimes E_i,\\
&\,& \Delta_h(F_i)=E_i \otimes \exp(\frac{hd_i H_i}{4})+ \exp(\frac{-hd_i H_i}{4})\otimes F_i,\\
&\,& \epsilon_h (H_i)=\epsilon_h(E_i)=\epsilon_h(F_i)=0.
\end{eqnarray*}
There exists a universal $r$-matrix $R_h$, which can be written explicitly, making  $\mathsf{U}_h(\g)$ a topological quasi-triangular bialgebra. Moreover, there is a canonical isomorphism of algebras:

\[ \frac{\mathsf{U}_h(\g)}{ h (\mathsf{U}_h(\g))}\cong \U.\]
\end{theorem}
\begin{proof}
The proof of this theorem can be found in \cite{Drinfeld2}.
\end{proof}

Given a finite dimensional representation $V$ of the complex semisimple Lie algebra $\g$, there exists a corresponding representation of the Drinfeld-Jimbo algebra $\mathsf{U}_h(\g)$. This fact can be deduced from the following result:

\begin{theorem}
Let $\g$ be a finite dimensional complex semisimple Lie algebra with Drinfeld-Jimbo algebra $\mathsf{U}_q(\g)$.
There exists an isomorphism of topological algebras:
\[ \varphi: \mathsf{U}_h(\g) \rightarrow \U[[h]],\]
which is the identity modulo $h$. Moreover, given any other such isomorphism $\varphi'$ there exists an invertible element $F \in \U[[h]]$,
congruent to $1$ modulo $h$, such that for any $x \in \mathsf{U}_h(\g)$:
\[ \varphi'(x)= F \varphi(x) F^{-1}.\]
\end{theorem}
\begin{proof}
The proof of this theorem can be found in Kassel's book \cite{Kassel}.
\end{proof}

The theorem above implies that given a complex finite dimensional representation $V$ of $\g$, the $\mathbb{C}[[h]]$-module $V[[h]]$ is a representation of $\mathsf{U}_h(\g)$, well defined up to isomorphism. Indeed, $V[[h]]$ is a module over $\U[[h]]$ and we can choose the isomorphism $\varphi$ as above to make it a module over the Drinfeld-Jimbo algebra. Since $\varphi$ is unique up to conjugation, the isomorphism class of the representation is well defined.
Because $\mathsf{U}_h(\g)$ is a topological quasi-triangular bialgebra, the universal $r$-matrix $R_h \in \mathsf{U}_h(\g)\hat{\otimes} \mathsf{U}_h(\g)$ defines an $r$-matrix on $V[[h]]\hat{\otimes} V[[h]]$. We conclude that for each $n\geq 2$
the $\mathbb{C}[[h]]$-module $V[[h]]^{\hat{\otimes}n}\cong V^{\otimes n}[[h]]$ is a representation of the braid group $B_n$ by $\mathbb{C}[[h]]$-linear automorphisms. We denote this representation by 
\[ \rho_{DJ}: B_n \rightarrow \Aut_{\mathbb{C}[[h]]}(V^{\otimes n}[[h]]).\]

For a finite dimensional representation $V$ of a complex semisimple Lie algebra $\g$, there are two different constructions of representations of the braid groups. On the one hand, the monodromy of the Knizhnik-Zamolodchikov connections gives the vector space $V^{\otimes n}$ the structure of a representation of $B_n$. On the other hand, the universal $r$-matrix in the Drinfeld-Jimbo algebra $\mathsf{U}_h(\g)$ provides a representation of $B_n$ on $V^{\otimes n}[[h]]$. The relationship between these two seemingly unrelated constructions is provided by the Drinfeld-Kohno theorem. For a complex number $h$,  let $\theta_{n, h}$ be the Knizhnik-Zamolodchikov connection with parameter $h$ defined by:
\[ \theta_{n,h}:= \frac{h}{2\pi i}\theta_n.\]

Since $d \theta_n=0= [\theta_n, \theta_n]$, the connection $\theta_{n,h}$ is flat. As before, for each $h \in \mathbb{C}$ we obtain a representation of the braid group $\rho_h:B_n \rightarrow V^{\otimes n}$. The monodromy of the connection $\theta_{n,h}$ can be computed explicitly in terms of iterated integrals and in particular, depends on $h$ in an analytic manner. Therefore, by taking Taylor series with respect to the parameter $h$, we obtain a representation:
\[ \rho_{KZ}: B_n \rightarrow \Aut_{\mathbb{C}[[h]]}(V^{\otimes n}[[h]]).\]

\begin{theorem}[Drinfeld \cite{Drinfeld} - Kohno \cite{Kohno3}]
Let $\g$ be a finite dimensional complex semisimple Lie algebra and $V$ a finite dimensional complex representation.
For each $n \geq 2$, the Knizhnik-Zamolodchikov representation of of the braid group $B_n$:

\[ \rho_{KZ}: B_n \rightarrow \Aut_{\mathbb{C}[[h]]} (V^{\otimes n}[[h]]),\]
obtained by taking Taylor series on the monodromy of the Knizhnik-Zamolodchikov connection
is equivalent to the Drinfeld-Jimbo representation:
\[ \rho_{DJ}: B_n \rightarrow \Aut_{\mathbb{C}[[h]]}(V^{\otimes n}[[h]]),\]
obtained via the universal $r$-matrix of the Drinfeld-Jimbo algebra $\mathsf{U}_h(\g)$.
\end{theorem}
\begin{proof}
The interested reader may refer to Kassel's book \cite{Kassel}.
\end{proof}

\bibliography{RCMBibTeX}
\end{document}